\documentclass[final, 12pt]{colt2025} 
\usepackage{times}

\usepackage{graphicx}
\graphicspath{{}{figs/}{../simulation/figs/}}
\usepackage{hyperref}
\usepackage{url}
\usepackage{enumitem}
\usepackage{xcolor}
\usepackage{color}
\usepackage{bm}
\usepackage[T1]{fontenc}
\usepackage[utf8]{inputenc}
\makeatother
\usepackage{graphicx} 
\usepackage{algorithm}
\usepackage{algorithmic}
\usepackage{comment}
\usepackage{amsmath,amssymb}
\usepackage{amsfonts} 
\usepackage{bbm}
\usepackage{mathtools}
\usepackage{bbm}
\usepackage{dsfont}
\usepackage{xifthen}

\newtheorem{assumption}{Assumption}

\usepackage{xcolor}
\usepackage{color}

\makeatother
\newcommand{\bluetext}[1]{{\leavevmode\color{blue}#1}}

\newcommand{\mcl}{\mathcal}
\newcommand{\Pmat}[1]{\mathbf{P}^{#1}}
\newcommand{\Pnext}{\mathbf{P}}

\newcommand{\Rvec}[1]{\mathbf{r}^{#1}}
\newcommand{\sspace}{\mcl{S}}

\newcommand{\bS}{\textbf{S}}
\newcommand{\bs}{\textbf{s}}
\newcommand{\bA}{\textbf{A}}
\newcommand{\ba}{\textbf{a}}
\newcommand{\pspace}{\Pi^{(N)}}
\newcommand{\pspacedyn}{\Bar{\Pi}}

\newcommand{\y}[2]{y_{#1, #2}}
\newcommand{\yb}{\textbf{y}}

\newcommand{\mB}{\textbf{x}}

\newcommand{\Dels}{\Delta_{\sspace}}

\newcommand{\syncconst}{\rho}

\newcommand{\THor}{T}

\newcommand{\Vlp}[2]{V_{\text{LP -}#1}^{(#2)}}
\newcommand{\Vlpt}[3]{V_{\text{LP -} #2}^{(#3)}(#1)}
\newcommand{\VoptInf}[2]{W_{#1}{(#2)}}
\newcommand{\Vopt}[1]{V_{\text{opt}}^{#1}}
\newcommand{\Vpol}[1]{V_{\pi}^{(#1)}}

\newcommand{\Vdetdisc}[2]{
    \ifthenelse{\isempty{#2}}{V^{\infty}_{\beta}(#1)}
    {V^{\infty}_{#2}(#1)}
}
\newcommand{\LPpol}[2]{\mu_{#1}(#2)}

\newcommand{\E}{\mathbb{E}}
\newcommand{\ProbP}{\mathbb{P}}

\newcommand{\R}{\mathbb{R}}

\newcommand{\one}{\textbf{1}}
\newcommand{\rotcost}[3]{\Tilde{#1}(#2, #3)}

\newcommand{\bx}{{\bm{x}}}

\newcommand{\by}{{\bm{y}}}

\newcommand{\bu}{\bm{u}}
\newcommand{\bU}{\bm{U}}
\newcommand{\bv}{\bm{v}}

\newcommand{\bZ}{\bm{Z}}
\newcommand{\Costmin}[2]{L_{#1}(#2)}
\newcommand{\bX}[1]{{\bm{X}(#1)}}

\def\mH{{\bm{H}}}

\def\mU{{\bm{U}}}

\def\mX{{\bm{X}}}

\newcommand{\laghomosub}{\mathbf{\nu_{\text{rel}}}}
\newcommand{\lagmax}{\mathbf{\nu_{\text{max}}}}
\newcommand{\biasvec}[1]{\mathbf{b}_{\text{rel}}(#1)}
\newcommand{\biasgen}[2]{\mathbf{b}_{#2}(#1)}
\newcommand{\gainrel}{\mathbf{g}_{\text{rel}}}
\newcommand{\gaingen}[1]{\mathbf{g}_{#1 }}

\def\floor#1{\lfloor #1 \rfloor}

\title{Model predictive control is almost optimal for restless bandits}

\coltauthor{%
 \Name{Nicolas Gast} \Email{nicolas.gast@inria.fr}\\
 \addr Univ. Grenoble Alpes, Inria, CNRS, Grenoble INP\thanks{Institute of Engineering Univ. Grenoble Alpes}, LIG, 38000 Grenoble, France%
 \AND
 \Name{Dheeraj Narasimha} \Email{dheeraj.narasimha@inria.fr}\\
 \addr Univ. Grenoble Alpes, Inria, CNRS, Grenoble INP, LIG, 38000 Grenoble, France%
}

\begin{document}

\maketitle

\begin{abstract}
    We consider the discrete time infinite horizon average reward restless markovian bandit (RMAB) problem. We propose a \emph{model predictive control} based non-stationary policy with a rolling computational horizon $\tau$. At each time-slot, this policy solves a $\tau$ horizon linear program whose first control value is kept as a control for the RMAB. Our solution requires minimal assumptions and quantifies the loss in optimality in terms of $\tau$ and the number of arms, $N$. We show that its sub-optimality gap is $O(1/\sqrt{N})$ in general, and $\exp(-\Omega(N))$ under a local-stability condition. Our proof is based on a framework from dynamic control known as \emph{dissipativity}. Our solution is easy to implement and performs very well in practice when compared to the state of the art. Further, both our solution and our proof methodology can easily be generalized to more general constrained MDP settings and should thus be of great interest to the burgeoning RMAB community.
\end{abstract}


\section{Introduction}
This work investigates the sequential decision making problem of Restless Multi-Armed Bandits (RMAB for short) over an infinite discrete time horizon. In this problem there are $N$ statistically identical arms. At each time step, the decision maker must choose for each arm if they would like to pull the arm or leave it as is. The decision maker has a constraint $\alpha N$ on the maximal number of arms that they may pull at each time instance. Each arm has a known state belonging to a common finite state space and upon choosing an action, produces a known state and action dependent reward. Next, the arms evolve to a new state independently according to a known state-action dependent transition kernel. These arms are only coupled through the budget constraint on the number of arms that may be pulled at each time instance. The state and reward are both revealed to the decision maker before the next decision needs to be made. The objective of the decision maker is to maximize the long-term time average reward.

This problem was first proposed by \citet{Wh88}. Over the years RMABs have been used to model a number of practical problems. These applications include web-crawling, queuing, communication systems, scheduling problems and many more, \citep{veatch1996scheduling}, \citep{dance2019optimal}, \citep{nino2002dynamic}, \citep{MearaWebcrawl01}. The problem of choosing a subset of tasks to perform among a larger collection of tasks under resource constraints shows up time and time again in various resource constrained control problems. For a comprehensive review on RMABs and their applications, the interested reader is directed towards \citet{NinoMora23}. While the existence of an optimal policy for RMABs is straightforward, \citet{PT99} showed that the exact solution to this problem is PSPACE-hard. Consequently, most work focused on designing approximate solutions with good performance guarantees. 

In the seminal paper, \citet{Wh88} suggested that under a condition known as indexability, an index can be associated with each state. This index is now referred to as the Whittle's Index (WI), and it was conjectured that a priority policy based on this index would be an optimal solution for this problem.  This setting naturally lends itself to mean field approximations where one may replace the $N$ armed problem with a dynamical system in order to find these approximate solutions. \citep{WW90} were the first to point out that under a (hard to verify) condition on the dynamical system known as uniform global attractor property (UGAP), the (WI) was asymptotically optimal. Recently, many of the results in the RMAB literature have focused on two major aspects of the problem: how quickly do proposed asymptotically optimal policies converge to the optimal solution as a function of the number of arms and whether the underlying assumptions such as indexability and UGAP can be made less restrictive to obtain more general conditions under which optimal solutions can be found. In the former category, under the assumptions of \citet{WW90}, \citet{GGY23} showed that the WI policy is exponentially close to the optimal solution. Recently, several works have been able to show an exponential order of convergence, \citet{GGY23b, HXCW24} without an indexability condition. On the other hand, in the latter category, \citet{HXCW23} showed an asymptotic convergence result under a less restrictive assumption known as the \emph{synchronization assumption}. Similar works include \citet{yan2024} and \citet{hong2024unichain} which show asymptotically optimal algorithms. Furthermore, \citet{HXCW24} loosen the restrictions to show asymptotic convergence and describe fundamental conditions to achieve exponential convergence rates for any algorithm. Our paper falls partially in the latter category by providing asymptotic convergence results under \emph{the weakest assumptions (to the best of our knowledge)} but we also show that under certain local stability assumptions one retrieves exponential convergence rates. A more comprehensive view of recent results can be found in the supplementary material.

\paragraph{Main Contributions}
The main contribution of our paper is the proof that a very natural model predictive control (LP-update) provides the \emph{best of both worlds}, it not only provides an algorithm that works well in practice, but we can also theoretically guarantee \emph{under the weakest assumptions (to the best of our knowledge) convergence to the optimal policy}. This idea of resolving an LP for a finite-horizon has been used previously; however, its use has been entirely restricted to either the finite-horizon case or, when used in the infinite horizon case, required a critical UGAP assumption(\emph{uniform global attractor property}) on the system. This assumption is known to be extremely hard to verify \cite{GGY23}. In this paper, by introducing the framework of \emph{dissipativity} we are able to show that a finite horizon policy indeed provides an almost optimal policy for the infinite horizon problem. Note, in the absence of such a framework, the finite horizon policy may veer very far from the infinite horizon policy \cite{DTGLS14}. The use of this framework is one of the key technical novelties of our approach. 
Apart from the main result, we make several fresh observations about the nature of RMAB problems:
\begin{itemize}[nosep]
    \item Our proof is of independent interest since it uses a new framework known as \emph{dissipativity}. Dissipativity is a closely studied phenomenon in the model predictive literature and is used to study how a policy drives a system towards optimal fixed points \citet{DTGLS14}. 
    \item Returning to the dynamics around the fixed point, we can tighten the rate of convergence to $e^{-cN}$ under a local stability condition.
    
    \item Perhaps the most helpful portion of our results, for practical purposes: the MPC algorithm works well in practice and is easy to implement. It performs well both in terms of the number of arms $N$ as well as the computational time horizon $T$, beating state-of-the-art algorithms in our bench marks. 
    \item Finally, from a control systems point of view, we return from trying to steer a dynamical system towards a fixed point to trying to maximize a value function. Policies can often be discontinuous in states while value functions are often smoother under mild assumptions.
\end{itemize}

\paragraph{Road-map} The rest of the paper is organized as follows. We describe the system model and the corresponding linear relaxation in Section~\ref{sec:model}. We build the LP-update algorithm in Section~\ref{sec:algo} and present its performance guarantee in Section~\ref{sec:main-theory}.  We provide the main ingredients of the proof in Section~\ref{sec:proof} postponing the most technical lemmas to the appendix.  We illustrate the performance of the algorithm in Section~\ref{sec:numerical}. The appendix contains additional literature review \ref{apx:review}, details of the algorithm and their extension to multi-constraints MDPs \ref{apx:algo}, additional proofs \ref{apx:PFcomp}, \ref{apx:PFcomp_expo} and details about the numerical experiments \ref{apx:parameters}. 

\paragraph{Reproducibility} The code to reproduce this paper (including all simulations and figures) is available at \url{https://github.com/ngast/paper_banditMPC_colt2025}.

\section{System Model and linear relaxation}
\label{sec:model}
\subsection{System Model}
We consider an infinite horizon discrete time restless Markovian bandit problem parameterized by the tuple $\langle \sspace, \Pmat{0}, \Pmat{1}, \Rvec{0}, \Rvec{1}; \alpha, N \rangle$.  A decision maker is facing $N$ statistically identical arms, and each arm has a state that belongs to the finite state-space $\sspace$. At each time-instant, the decision maker observes the states of all arms $\bs=\{s_1\dots s_N\}$ and chooses a vector of actions $\ba=\{a_1\dots a_N\}$, where the action $a_n=1$ (respectively $a_n=0$) corresponds to pulling the arm $n$ (or leaving it). The decision maker is constrained to pull at most $\alpha N$ arms at each decision epoch. 

The matrices $\Pmat{0}$ and $\Pmat{1}$ denote the transitions matrices of each arm and the vectors $\Rvec{0}, \Rvec{1}$ denote the $|\sspace|$ dimensional vector for the rewards. We assume that all the rewards $\Rvec{}$ lie between $0$ and $1$. As the state-space is finite, this assumption can be made without loss of generality by scaling and centering the reward vector. 
We assume that the transitions of all arms are Markovian and independent. This means that if the arms are in state $\bs := \{s_1, s_2 \dots s_N\}\in\sspace^N$ and the decision maker takes an action $\ba\in\{0,1\}^N$, then the decision maker earns a reward $\frac{1}{N}\sum_{n} r^{a_n}_{s_n}$ and the next state becomes $\bs':= \{s'_1, s'_2 \dots s'_N\}\in\sspace^N$ with probability
\begin{equation}\label{EQ:MKEVOL}
    \begin{split}
            &\ProbP(\bS(t + 1) = \bs'|\bS(t) = \bs,\bA(t) = \ba, \dots \bS(0), \bA(0)) \\
            &\qquad= \ProbP(\bS(t + 1) = \bs'|\bS(t) = \bs,\bA(t) = \ba)
            = \Pi_{n = 1}^{N} P^{a_n}_{s_n, s_n'}.
    \end{split}
\end{equation}
It is important to note that the arms are only coupled through the budget constraint $\alpha$.

Let $\pi$ be a stationary policy mapping each state to a probability of choosing actions, i.e, $\pi : \sspace^{N} \to \Delta (\{0, 1\}^{N})$ subject to the budget constraint $\alpha$. Let $\pspace$ denote the space of all such policies. We define the average gain of policy $\pi\in\pspace$ as
\begin{equation}\label{EQ::VPOLT}
    \Vpol{N} = \lim_{\THor \to \infty}\frac{1}{\THor}\E_{\pi}\left[ \sum_{t = 0}^{\THor - 1} \frac{1}{N} \sum_{n = 1}^N r^{A_n(t)}_{S_n(t)}\right].
\end{equation}
In theory, the average gain of a policy might depend on the initial state $\bS^{(N)}(0)$. Yet, under mild conditions (which will be verified in our case), this value does not depend on the initial state. This is why we omit the dependence on the initial state and simply write $\Vpol{N}$.
 
Here $(S_n(t), A_n(t))$ denotes the state-action pair of the $n$th arm at time $t$ and $r^{A_n(t)}_{S_n(t)}$ denotes the $S_n(t)^{\text{th}}$ entry of the $\Rvec{A_n(t)}$ vector. As the state-space and action space is finite, for any stationary policy $\pi$, the limit is well defined, \cite{puterman2014markov}. The RMAB problem amounts to computing a policy $\pi$ that maximizes the infinite average reward. We denote the optimal value of the problem as
\begin{align}\label{EQ::VOPTT}
    \Vopt{N} &:= \max_{\pi \in \pspace} \Vpol{N}.
\end{align}
The optimal policy exists and the limit is well defined, \citet{puterman2014markov}. 

\subsection{Alternative state representation via empirical distribution}

In order to build an approximation of \eqref{EQ::VOPTT}, we introduce an alternative representation of the state space, that we will use extensively in the paper. Given any joint state of the arms $\bs\in\sspace^{N}$, we denote the empirical distribution of these arms as $\mB(\bs)\in\Delta_S$, where $\Dels$ is the simplex of dimension $|\sspace|$. $\mB(\bs)$ is a vector with $|\sspace|$ dimensions and $x_i(\bs)$ is the fraction of arms that are in state $i$.  Next, given an action vector $\ba$, we denote by $\bu(\bs, \ba)$ the empirical distribution of the state-action pairs ($s,1$). In words, $u_i(\bs,\ba)$ is the fraction of arms that are in state $i$ and that are pulled.  Since no more than $N\alpha$ arms can be pulled at any time instance and no more than $Nx_i(\bs)$ arms can be pulled in state $i$, it follows that when $\mB$ is fixed, $\bu$ satisfies the following inequalities,
    \begin{equation}\label{EQ::POL}
    0 \leq \bu \leq \mB \hspace{0.3 in} \|\bu\|_1 \leq \alpha \|\mB\|_1 = \alpha,
\end{equation}
where $\bu\leq\mB$ denotes a component-wise inequality and $\|\cdot \|_1$ denotes the $l_1$ norm. We denote by $\mcl{U}(\mB)$ the set of feasible actions for a given $\mB$, \emph{i.e.}, the set of $\bu$ that satisfy \eqref{EQ::POL}.

\subsection{Linear relaxation}

We consider the following linear program:
\begin{subequations}
    \label{EQ::VOPTDYN-Tinf}
    \begin{align}
       \Vopt{\infty} := \max_{\bx,\bu\in\Delta_{\sspace}}&~ \Rvec{0}\cdot \mB + (\Rvec{1} - \Rvec{0})\cdot\bu,\label{EQ:VOPTDYN-inf1}\\
        \text{Subject to: }\quad
        &\bu \in \mcl{U}(\mB) \label{EQ::MF3-inf}\\
        &\mB = \mB  \Pmat{0} +  \bu (\Pmat{1} - \Pmat{0}) \label{EQ::MF2-inf}
    \end{align}
\end{subequations}
This linear program is known to be a relaxation of \eqref{EQ::VOPTT} that is asymptotically tight, that is $\Vopt{N}\le \Vopt{\infty}$ for all $N$ and $\lim_{N\to\infty}\Vopt{N}=\Vopt{\infty}$, see \citet{GGY23b,HXCW24}.

To give some intuition on the relationship between \eqref{EQ::VOPTT} and \eqref{EQ::VOPTDYN-Tinf}, we remark that if $\bX{t} := \mB(\bS^{N}(t))$ is the empirical distribution of states at time $t$ and $\bU(t)=\bu(\bS(t),\bA(t))$ is the joint control, then it is shown in \citet{GGY23b} that the Markovian evolution \eqref{EQ:MKEVOL} implies
\begin{equation}
\E[\bX{t + 1} \mid \bX{t}, \bU(t)] = \mX{(t)}  \Pmat{0} +  \bU(t)(\Pmat{1} - \Pmat{0}).
\label{EQ:evolution_mf}
\end{equation}
In \eqref{EQ::VOPTDYN-Tinf}, the variable $x_i$ corresponds to the time-averaged fraction of arms in state $i$; similarly the variable $u_i$ corresponds to the time-averaged fraction of arms in state $i$ that are pulled. The constraint \eqref{EQ::MF3-inf} imposes that \emph{on average}, no more than $\alpha N$ arms are pulled. This is in contrast with the condition imposed for problem \eqref{EQ::VOPTT} that enforces this condition at each time step.

\section{Construction of the LP-update policy}
\label{sec:algo}

\subsection{The Finite-Horizon Mean Field Control Problem}

To build the LP-update policy, we consider a \emph{controlled dynamical system}, also called the \emph{mean field model}, that is a finite-time equivalent of \eqref{EQ::VOPTDYN-Tinf}. For a given initial condition $\bx(0)$ and a time-horizon $\tau$, the states and actions of this dynamical system are constrained by the evolution equations 
\begin{subequations}
    \label{EQ:MeanField}
    \begin{align}
        \bu(t) &\in \mcl{U}(\bx(t))  \label{EQ::MF3} \\
        \bx({t + 1}) &= \bx(t) \Pmat{0}  + \bu(t) (\Pmat{1} - \Pmat{0}) , \label{EQ::MF2}
    \end{align}
\end{subequations}
$\forall t\in\{0\dots \tau-1\}$. In the above equation, \eqref{EQ::MF2} should be compared with  \eqref{EQ:evolution_mf} and indicates that $\bx(t)$ and $\bu(t)$ correspond to the quantities $\E[\bX{t}]$ and $\E[\bU(t)]$ of the original stochastic system. As the constraint \eqref{EQ::MF3} must be ensured by $\bx(t)$ and $\bu(t)$, this constraint \eqref{EQ::POL} must be satisfied for the expectations: $\E[\bU(t)]\in\mcl{U}(\E[\bX{t}])$.

The reward collected at time $t$ for this dynamical system is $\Rvec{0}\cdot \bx(t) + (\Rvec{1} - \Rvec{0})\cdot\bu(t)$. Let $\lambda$ be the dual multiplier of the constraint \eqref{EQ::MF2-inf} of an optimal solution of \eqref{EQ::VOPTDYN-Tinf}. We define a \emph{deterministic} finite-horizon optimal control problem as:
\begin{subequations}
    \label{EQ::VOPTDYN}
    \begin{align}
       \VoptInf{\tau}{\mB(0)} &= \max_{\bx, \bu} \sum_{t = 0}^{\tau - 1} \left( \Rvec{0}\cdot \mB(t) + (\Rvec{1} - \Rvec{0})\cdot\bu(t)\right) + \lambda \cdot\bx(\tau),\label{EQ:VOPTDYN1}\\
         &\text{Subject to:} \text{ $\bx$  and $\bu$ satisfy \eqref{EQ:MeanField} for all $t\in\{0,\tau-1\}$,}
    \end{align}
\end{subequations} 
Before moving forward, the equation above deserves some remarks. First, for any finite $\tau$, the objective and the constraints \eqref{EQ:MeanField} for the optimization problem \eqref{EQ::VOPTDYN} are linear in the variables $(\bx(t),\bu(t))$. This means that this optimization problem is computationally easy to solve. \textbf{In what follows, we denote by $\LPpol{\tau}{\mB}$ the value of $\bu(0)$ of an optimal solution to \eqref{EQ::VOPTDYN}.}

Second, the definition of \eqref{EQ::VOPTDYN} imposes that the constraint $\|\bu\|_1 \leq \alpha \|\mB\|_1 = \alpha $ holds for each time $t$. This is in contrast to the way this constraint is typically treated in RMAB problems, in which case \eqref{EQ::POL} is replaced with the time-averaged constraint $\frac{1}{T}\sum_{t = 0}^{T - 1}\|\bu(t)\|_1 \leq \alpha$. The latter relaxation was introduced in \cite{Wh88} and is often referred to as Whittle's relaxation \citep{Avrachenkov2020WhittleIB,Avrachenkov2021}. This is the constraint that we use to write  \eqref{EQ::VOPTDYN-Tinf}.  \cite{GGY23} showed that for any finite $T$, the finite-horizon equivalent of \eqref{EQ::VOPTT} converges to \eqref{EQ::VOPTDYN} as $N$ goes to infinity. The purpose of this paper is to show that the solution of the finite $T-$horizon LP \eqref{EQ::VOPTDYN} provides an almost-optimal solution to the original $N-$arm average reward problem \eqref{EQ::VOPTT}. 

Last, as we will discuss later, taking $\lambda$ as the dual multiplier of the constraint \eqref{EQ::MF2-inf} helps to make a connection between the finite and the infinite-horizon problems \eqref{EQ::VOPTDYN} and \eqref{EQ::VOPTDYN-Tinf}. Our proofs will hold with minor modification by replacing $\lambda$ by $0$ and in practice we do not use this multiplier at all.

\subsection{The Model Predictive Control Algorithm}
\begin{algorithm}[ht]
	\caption{Evaluation of the LP-Update policy}
	\label{algo::MPC}
	\begin{algorithmic}
            \REQUIRE Horizon $\tau$, Initial state $\bS^{(N)}(0)$, model parameters $\langle \Pmat{0}, \Pmat{1}, \Rvec{0}, \Rvec{1} \rangle$, and time horizon $T$
            \STATE Total-reward $\leftarrow 0$.
            \FOR{$t=0$ to $T-1$}
            \STATE $\bu(t) \leftarrow \LPpol{\tau}{\mB(\bS^{(N)}(t))}$.
            \STATE $\bA^{(N)}(t)$ $\leftarrow$ Randomized Rounding $(\bu(t))$ 
            (by using Algorithm~\ref{algo::rounding}).
            \STATE Total-reward $\leftarrow$ Total-reward + $R(\bS^{(N)}(t), \bA^{(N)}(t))$.
            \STATE Simulate the transitions according to \eqref{EQ:MKEVOL} to get $\bS^{(N)}(t + 1)$            \ENDFOR
            \ENSURE{$\text{Average reward}: \frac{\text{Total-reward}}{\THor} \approx \Vlp{\tau}{N}$}
	\end{algorithmic}
\end{algorithm}

The pseudo-code of the LP-update policy is presented in Algorithm \ref{algo::MPC}. The LP-update policy takes as an input a time-horizon $\tau$. At each time-slot, the policy solves the finite horizon linear program \eqref{EQ::VOPTDYN} to obtain $\LPpol{\tau}{\mB}$ that is the value of $\bu(0)$ for an optimal solution to \eqref{EQ::VOPTDYN}. Note that such a policy may not immediately translate to an applicable policy as we do not require that $N\LPpol{\tau}{\mB}$ be integers. We therefore use \emph{randomized rounding} to obtain a feasible policy for our $N$ armed problem, $\bA^{N}(t)$. Applying these actions to each arm gives an instantaneous reward and a next state. This form of control has been referred to as \emph{rolling horizon} \citet{puterman2014markov} but is more commonly referred to as \emph{model predictive control} \citet{DTGLS14}. Our algorithm may be summarized by:
\begin{equation}\label{EQ::ALGO}
    \bS^{(N)}(t) \xrightarrow[]{\text{Solve LP \eqref{EQ::VOPTDYN}}}\LPpol{\tau}{\mB(\bS^{(N)}(t))} \xrightarrow[\text{rounding}]{\text{Randomized}} \bA^{N}(t) \xrightarrow[\text{new state}]{\text{Observe}} \bS^{(N)}(t + 1). 
\end{equation}
We use a randomized rounding procedure similar to \cite{GGY23}, 
see Appendix~\ref{apx:rounding}.

\section{Main theoretical results}
\label{sec:main-theory}

The main result of our paper is to show that a finite-horizon model predictive control algorithm, that we call the LP-update policy, is asymptotically optimal for the infinite horizon bandit problem. Note that this LP-update policy is introduced in \citet{GGY23,GGY23b,ghosh2022indexability} for finite-horizon restless bandit.

\subsection{First result: LP-update is asymptotic optimal}

Our result will show that the LP-update policy is asymptotically optimal as the number of arms $N$ goes to infinity, under an easily verifiable mixing assumption on the transition matrices. To express this condition, for a fixed integer and a sequence of actions $\ba=(a_1\dots a_k)\in\{0,1\}^k$, we denote by $P^{\ba}_{i, j}$ the $(i,j)$th entry of the matrix $\prod_{t=1}^k \Pmat{a_k}$. We then denote by $\rho_k$ the following quantity\footnote{The definition \eqref{eq:ergodic_coef} is related to the notion of ergodic coefficient defined in \citep{puterman2014markov} that used a related constant to prove convergence of value iteration algorithms in span norm for the unconstrained discounted Markov Decision Process.}: 
\begin{align}
    \label{eq:ergodic_coef}
      \syncconst_{k} \triangleq \min_{s, s'\in\sspace, \ba\in\{0,1\}^k} \sum_{s^* \in \sspace}\min\{P^{a_1, a_2, \dots a_k}_{s, s^*}, P^{0, 0, \dots 0}_{s', s^*}\}    
\end{align}
In the above equation, the minimum is taken over all possible initial states $s,s'$ and all possible sequence of actions. The quantity $\syncconst_k$ can be viewed as the probability (under the best coupling) that two arms starting in states $s$ and $s'$ reach the same state after $k$ iterations, if the sequence $a_1\dots a_k$ is used for the first arm while the second arm only uses the action $0$.  The assumption that we use for our result is that $\syncconst_k>0$ for some integer $k$.
\begin{assumption}\label{AS::SA}
    There exists a finite $k$ such that $\syncconst_{k} > 0$.
\end{assumption}
While the assumption may look abstract, note that when the $\Pmat{0}$ matrix is ergodic, it ensures that assumption \ref{AS::SA} holds. Indeed in this case, there exists a $k>0$ such that $P^{0\dots 0}_{ij}>0$ for all $i,j$ which would imply that $\syncconst_{k} > 0$.  Related assumptions and their relationship to Ergodicity can be found in \citet{hernandez2012discrete}.  Assumption~\ref{AS::SA} is similar to the unichain and aperiodic condition imposed in \citet{hong2024unichain}.  Note that this quantity involves the best coupling and not a specific coupling; it is more general than the synchronization assumption from \citet{HXCW23}. \\
We are now ready to state our first theorem, in which we provide a performance bound of the average reward of the LP-update policy, that we denote by $\Vlp{\tau}{N}$.

\begin{theorem}
\label{thm:asymptotic_optimal}
    Assume \ref{AS::SA} with ergodicity constant $\syncconst_{k}$ for some fixed integer $k$. Fix a positive constant $\epsilon > 0$;  then there exists $\tau(\epsilon)$ such that, Algorithm~\ref{algo::MPC} has the following guarantee of performance:
    \begin{equation}
        \Vlp{\tau(\epsilon)}{N} \ge \Vopt{\infty} - 2\epsilon  - \left(\frac{2k}{\syncconst_k}\left[3 +  \frac{2k\alpha}{\syncconst_k}\right] + 1\right)\frac{\alpha N - \lfloor{\alpha N}\rfloor}{N} - \frac{k}{\syncconst_k}\left(3 +  \frac{2k\alpha}{\syncconst_k}\right)
        \sqrt{\frac{|\mcl{S}|}{N}}
    \end{equation}
\end{theorem}
As $\Vopt{N}\le \Vopt{\infty}$, this result shows that the LP-update policy becomes optimal as $\tau$ and $N$ go to infinity. The sub-optimality gap of LP-update decomposes in three terms. The first term corresponds to an upper-bound on the sub-optimality of using the finite-horizon $\tau$ when solving the LP-problem \eqref{EQ::VOPTDYN}. Our proof shows that one can take $\tau(\epsilon) = \mcl{O}(\frac{1}{\epsilon})$, and in the numerical section, we will show that choosing a small value like $\tau=10$ is sufficient for most problems. The second term corresponds to a rounding error: for all $\alpha$ such that $N\alpha $ is an integer, this term equals $0$.  The dominating error term is the last term, $O(1/\sqrt{N})$. It corresponds to the difference between the $N$-arm problem and the LP-problem with $\tau=+\infty$. 

\subsection{Second result: Exponentially small gap under a stability condition}

Theorem~\ref{thm:asymptotic_optimal} shows that the sub-optimality gap of the LP-update policy is of order $O(1/\sqrt{N})$ under quite general conditions. While one could not hope for a better convergence rate in general, there are known cases for which one can construct policies that become optimal exponentially fast when $N$ goes to infinity. This is the case for Whittle index under the conditions of indexability, uniform global attractor property (UGAP), non-degeneracy and global exponential stability \cite{GGY23}. More details can be found in Appendix \ref{apx:review}. 
In this section, we show that LP-update also becomes optimal exponentially fast under essentially the same conditions as the ones presented in \cite{HXCW24}. The first condition that we impose is that the solution of the above LP-problem is non-degenerate (as defined in \citet{GGY23,HXCW24}). 
\begin{assumption}[Non-degenerate]
    \label{AS::non-degenerate}
    We assume that the solution ($\bx^*,\bu^*)$ to the linear program \eqref{EQ::VOPTDYN-Tinf} is unique and satisfies that $x^*_i>0$ for all $i\in\sspace$ and that there exists a (unique) state $i^*\in\sspace$ such that $0 < u^*_{i^*} < x^*_{i^*}$.
\end{assumption}

The second condition concerns the local stability of a map around the fixed point.
\begin{assumption}
    \label{AS::stable}
    Assume~\ref{AS::non-degenerate} and let $P^*$ be the $|\sspace|\times|\sspace|$ matrix such that 
    \begin{align*}
        P^*_{ij} = \left\{\begin{array}{ll}
            P^0_{ij} &\text{ if $i$ is such that $u^*_i<x^*_i$.}\\
            P^1_{ij}-P^{1}_{i^*j}+P^{0}_{i^*j} &\text{ if $i$ is such that $u^*_i=x^*_i$.}
        \end{array}\right.
    \end{align*}
    
    We assume that the matrix $\Pmat{*}$ is stable, \emph{i.e.}, that the $l_2$ norm of all but one of the Eigenvalues of $\Pmat{*}$ are strictly smaller than 1.
\end{assumption}
Both these conditions are equivalent to the assumption of non-degeneracy and local stability defined in \cite{HXCW24}.

The last condition that we impose is a technical assumption that simplifies the proofs. 
\begin{assumption}[Unicity]
    \label{AS::unique}
    We assume that for all $\bx\in\Delta_{\sspace}$, the LP program \eqref{EQ::VOPTDYN} has a unique solution.
\end{assumption}
This assumption guarantees that the LP-update policy is uniquely defined. Note that the assumptions of unicity of the fixed point are often made implicitly in papers when authors talk about ``the'' optimal solution instead of ``an'' optimal solution. The assumptions of \emph{non-degeneracy} and \emph{unicity} are not restrictive since, degenerate solutions essentially occupy a zero-measure set: if a problem is degenerate (or has multiple solutions), then adding a small noise to the parameters will make this problem non-degenerate (or will guarantee the uniqueness of the solution). The situation is, however, different for the local stability assumption: if Assumption~\ref{AS::stable} does not hold for a given problem, then it will not hold for any problem that is close enough to this problem. Moreover, we believe that this assumption is necessary to derive a tighter bound in terms of the number of arms (\cite{HXCW24}).

\begin{theorem}
    \label{thm:expo_bound}
    Assume \ref{AS::SA},  \ref{AS::non-degenerate}, \ref{AS::stable}, and \ref{AS::unique}; then there exist constants $C',C''>0$ (independent of $N$ and $\tau$) such that for all $\epsilon>0$, with $\tau(\epsilon)$ (set according to Theorem ~\ref{thm:asymptotic_optimal}) and $N$ such that $\alpha N$ is an integer, Algorithm \ref{algo::MPC} has the following guarantee of performance:
    \begin{equation}
        \Vlp{\tau(\epsilon)}{N} \ge \Vopt{\infty} - 2\epsilon - C'e^{-C'' N}.
    \end{equation}
\end{theorem}
The first term of the bound above is identical to the one used in Theorem~\ref{thm:asymptotic_optimal}. What is more important is that the last term  decays exponentially with $N$. The results of Theorem \ref{thm:asymptotic_optimal} meet, order-wise the best known lower bounds \emph{without local stability conditions} from \cite{HXCW24} whereas Theorem \ref{thm:expo_bound} meets the best exponential bounds \emph{with stability conditions.} Our theorem shows that the LP-update policy also benefit from similar performance guarantees (while performing better for small values of $N$).  Previous results on LP-update policies were only able to match our bounds under the UGAP assumptions. \emph{Our relaxation of these assumptions can be regarded as one of the main contributions of our works.}

The bounds that we obtain measure the performance gap $\Vlp{\tau(\epsilon)}{N}-\Vopt{\infty}$ which is the difference between the value of the LP-update policy for the system of size $N$ and the value of the relaxed LP. This performance gap is known to be of order at least $\Omega(1/\sqrt{N})$ for general degenerate problems and at least $\exp{-O{N}}$ for general non-degenerate problems.  As $\Vopt{\infty}\ge\Vopt{N}\ge$, this implies that the same bounds of $O(1/\sqrt{N})$ and $\exp(-\Omega(N))$ hold for the sub-optimality gap $\Vlp{\tau(\epsilon)}{N}-\Vopt{N}$.  In a recent paper, \cite{yan2024optimalgap} directly study this sub-optimality gap directly without comparing the benchmark to $\Vopt{\infty}$ for the finite horizon problem.

\section{Proofs: Main Ideas}
\label{sec:proof}
In this section, we provide the major ingredients of the proofs of the two main theorems. We provide more details for the proof of Theorem~\ref{thm:asymptotic_optimal} because this is the more original of the two. The proofs of all lemmas and some details of computation are deferred 
to the supplementary material.


\subsection{Sketch for Theorem \ref{thm:asymptotic_optimal}:}

Three major components are required in order to complete the proof. 

\paragraph{Part 1, Properties of the dynamical control problem \eqref{EQ::VOPTDYN}:} For $\bx,\bu$, we denote by $\Phi(\bx, \bu) := \bx\Pmat{0} + \bu (\Pmat{1} - \Pmat{0})$ the deterministic transition kernel, and we recall that the instantaneous reward is $R(\bx, \bu) := \Rvec{0}\cdot \bx + (\Rvec{1} - \Rvec{0})$. In Lemma~\ref{LEM::EXGH},
we establish several properties that relate the finite-horizon problem \eqref{EQ::VOPTDYN} and the finite-horizon problem \eqref{EQ::VOPTDYN-Tinf} that hold under Assumption~\ref{AS::SA}. First, we show that the average gain of the finite-time horizon problem \eqref{EQ::VOPTDYN} converges to the average gain of the infinite-horizon problem, that is $\lim_{\tau\to\infty} W_\tau(\bx)/\tau = g^*$ for all $\bx$. Second, we also show that the \emph{bias function} $h^{\star}(\cdot) : \Dels \to \R$ given by :
\begin{equation}\label{EQ::Defbias}
    h^{\star}(\bx) := \lim_{\tau \to \infty} \VoptInf{\tau}{\bx} - \tau g^{\star}
\end{equation} 
is well defined and Lipschitz-continuous with constant $k/\syncconst_k$ and the convergence in \eqref{EQ::Defbias} is uniform in $\bx$. Moreover, the gain and the bias function satisfy  
\begin{equation}\label{EQ::FIXEDPT}
    h^{\star}(\bx) + g^{\star} = \max_{\bu \in \mcl{U}(\bx)} R(\bx, \bu) + h^{\star}(\Phi(\mB, \bu)).
\end{equation} 
While both these definitions are well known in the average reward \emph{unichain MDP}, \emph{single arm} setting \emph{without constraints}, 
we establish these definitions for the \emph{constrained, deterministic problem}. Note the difference in the second term on the right hand side of the fixed point equation: typically the right hand side in the single arm problem takes the form $R(s, a) + \sum_{s'}P(s'|s, a) h(s')$, whereas here the expectation is inside rather than outside! 


\paragraph{Part 2, \emph{Dissipativity} and \emph{rotated cost}}
Let $(\bx^*, \bu^*)$ be the\footnote{For clarity we will present the proof as if  this point is unique although the proof also holds without this requirement.} optimal solution of the infinite-horizon problem \eqref{EQ::VOPTDYN-Tinf}, and let $l(\mB, \bu) := g^{\star} - \Rvec{0}\cdot \mB - (\Rvec{1} - \Rvec{0})\cdot\bu$. Following \cite{DTGLS14}, an optimal control problem with stage cost $l(\bx, \bu)$ and dynamic $\bx(t + 1) := \Phi(\bx, \bu)$ is called \emph{dissipative} if there exists a \emph{storage function} $\lambda : \Dels \to \R$ that satisfies the following equation:
\[
    \rotcost{l}{\bx}{\bu} := l(\mB, \bu) + \lambda(\mB) - \lambda(\Phi(\mB, \bu)) \geq l(\bx^*, \bu^*) = \rotcost{l}{\bx^*}{\bu^*} = 0.
\]
The cost, $\rotcost{l}{\mB}{\bu}$ is called the \emph{rotated cost function}.

In Lemma~\ref{LEM::DISS}, we show
that our problem is dissipative by setting the storage function $\lambda(x) := \lambda \cdot x $, where $\lambda$ is the optimal dual multiplier of the constraint \eqref{EQ:VOPTDYN1}. It is important to note, the rotated cost so defined is always non-negative.

\paragraph{Part 3, MPC is optimal for the deterministic control problem} By using our definition of rotated cost, we define the following minimization problem
\begin{align}
    \Costmin{\tau}{\mB} &:= \min\sum_{t = 0}^{\tau - 1}\rotcost{l}{\mB(t)}{\bu(t)}.\nonumber\\
    &\text{Subject to:} \text{ $\bx(t)$  and $\bu(t)$ satisfy \eqref{EQ:MeanField} for all $t\in\{0,\tau-1\}$.}
    \label{EQ::COSTMINPROB}
\end{align}
By dissipativity, $L_\tau(\bx)$ is monotone increasing. Moreover, we have that $\Costmin{\tau}{\mB} \geq 0 = \Costmin{\tau}{\bx^*}$. Hence, the problem is \emph{operated optimally at the fixed point}, $(x^*, u^*)$. The optimal operation at a fixed point is a key observation, made by several works, \citep{Avrachenkov2024, GGY23b, yan2024, HXCW23}, we recover this result as a natural consequence of dissipativity.
By Lemma~\ref{lem:C_vs_W},
$\Costmin{\tau}{\mB}= \tau g^{\star} - \VoptInf{\tau}{\bx} + \lambda \cdot \bx$ because\footnote{In fact, the two control problems are equivalent: $\bu^*$ is optimal for \eqref{EQ::VOPTDYN} if and only if it is for \eqref{EQ::COSTMINPROB}. } of a telescopic sum of the terms $\lambda \cdot \bx(t)$.  Combining this with \eqref{EQ::Defbias} implies that $\lim_{\tau\to\infty}\Costmin{\tau}{\bx} = h^{\star}(\bx) + \lambda \cdot \bx $. We will denote this limit by $\Costmin{\infty}{\bx}$. 
As $L_\tau(\bx)$ is monotone, it follows that, for any $\epsilon > 0$, there exists a $\tau(\epsilon)$ such that 
$|\Costmin{\infty}{\bx} - \Costmin{\tau(\epsilon)}{\bx}| < \epsilon$. Note, the monotonicity of $L_\tau(\bx)$ is crucial to obtain this property, leading us to use dissipativity. Putting the components together, we can now prove the final steps.

\paragraph{Proof of Theorem~\ref{thm:asymptotic_optimal}}

 We begin by considering the difference between the optimal value and the model predictive value function. We let $\bX{t}$ denote the empirical distribution of the arms at time $t$ and $\bU(t)$ be the corresponding empirical distribution of the actions. We drop the superscript $N$ for convenience in the proof below, but it should be noted that $\bU(t)$ is always obtained from a randomized rounding procedure and hence, is dependent on $N$. 
\begin{align}
\Vopt{N} - \Vlp{\tau}{N} 
&\le \lim_{T\to\infty}\frac{1}{T}\sum_{t = 0}^{T-1} \E \left[g^{\star} - R(\mX(t), \mU(t)) \right]\nonumber\\
&= \lim_{T\to\infty}\frac{1}{T}\sum_{t = 0}^{T-1} \E \left[g^{\star} - \Rvec{0}\cdot \mX(t) - (\Rvec{1} - \Rvec{0})\cdot\bu(t) +  (\Rvec{1} - \Rvec{0}) \cdot \left(\bu(t) - \bU(t)\right)\right]\nonumber\\
& \le \lim_{T \to \infty} \frac{1}{T}\sum_{t = 0}^{T-1} \E \left[l(\mX(t), \bu(t)) \right] + \frac{\alpha N - \lfloor{\alpha N}\rfloor}{N}\label{EQ:EQV3}
\end{align}
The first inequality follows from the well known result $\Vopt{N} \leq g^{\star}$, \citep{yan2024, GGY23, HXCW23}. The last inequality follows from randomized rounding 
and Lemma \ref{apx::lem_round}.
Let $(A):=\lim_{T\to\infty} \frac{1}{T}\sum_{t = 0}^{T-1} \E \left[l(\mX(t), \bu(t)) \right]$ denote the first term of \eqref{EQ:EQV3}. Adding and subtracting the storage cost we have:
\begin{align*}
    &(A)  = \lim_{T\to\infty} \frac{1}{T}\sum_{t = 0}^{T-1} \E \left[\rotcost{l}{\mX(t)}{\bu(t)} - \lambda\cdot\mX(t) + \lambda\cdot\Phi(\mX(t), \bu(t)) \right].
\end{align*}
Now note, by the dynamic programming principle $\rotcost{l}{\bx}{\bu} = \Costmin{\tau}{\bx} - \Costmin{\tau - 1}{\Phi(\bx, \bu)}$. Further, for any state $\bx$ and its corresponding control $\bu$ consider: $$\Costmin{\tau}{\bx} - \Costmin{\tau - 1}{\Phi(\bx, \bu)} = \Costmin{\tau}{\bx} - \Costmin{\infty}{\Phi(\bx, \bu)} + \Costmin{\infty}{\Phi(\bx, \bu)} - \Costmin{\tau - 1}{\Phi(\bx, \bu)}.$$
By choosing $\tau = \tau(\epsilon)$ we have, $\max \{|\Costmin{\tau - 1}{\bx} - \Costmin{\infty}{\bx}|, |\Costmin{\tau}{\bx} - \Costmin{\infty}{\bx}|\} < \epsilon$.
Plugging these inequalities together, introducing a telescopic sum, and manipulating the order of variables slightly (shown in Appendix 
~\ref{apx:proof_thm1_details}), the term $(A)$ is smaller than
\begin{align}
    \label{EQ:EQV4}
    2\epsilon + \lim_{T\to\infty}  \frac{1}{T}\sum_{t = 0}^{T - 1} \E \left[\Costmin{\infty}{\mX(t + 1)} - \Costmin{\infty}{\Phi(\mX(t), \bu(t))} - \lambda\cdot [\mX(t {+} 1) {-} \Phi(\mX(t), \bu(t))] \right].
\end{align}
By Lemma~1 of \cite{GGY23b}, we have $\|\E [\mX(t {+} 1)| \mX(t), \bU(t)] - \Phi(\mX(t), \bU(t))\|_1 \le \sqrt{|S|}/\sqrt{N}$. Moreover, our rounding procedure implies that $\|\Phi(\mX(t), \bU(t))-\Phi(\mX(t), \bu(t))\|_1\le C_\Phi(\alpha N - \lfloor\alpha N\rfloor)/N$ where $C_\Phi\le2$ is the Lipschitz-constant of the map $\Phi$ 
(see \ref{apx:C5}). With a few extra steps (see \ref{apx:C5})we get:
\begin{align}
    \label{EQ:EQV4b}
    (A) \le 2\epsilon + (C_{L} + \|\lambda\|_\infty) \left(\frac{\sqrt{|S|}}{\sqrt{N}} + 2\frac{\alpha N - \lfloor\alpha N\rfloor}{N} \right),
\end{align}
where $C_{L}$ is a Lipschitz constant for the limiting map $\Costmin{\infty}{\cdot}$ and $\|\lambda\|_\infty$ is the infinity-norm of the Lagrange multiplier $\lambda$. In Appendix 
Appendix \ref{apx:MPC}[Lemma \ref{lem:C_vs_W}],
we show that $C_{L}$ is bounded above by $k/\rho_k + \|\lambda\|_\infty$, further, in 
Appendix \ref{APX:LAGBOUND}
we show that $\|\lambda\|_\infty\le (k/\rho_k)(1+\alpha k/\rho_k)$.  The theorem follows by substituting these values into \eqref{EQ:EQV4b} and adding a term $(\alpha N - \lfloor{\alpha N}\rfloor)/N$ coming from \eqref{EQ:EQV3}. See Appendix 
~\ref{apx:C5}.

\subsection{Theorem~\ref{thm:expo_bound}}

The proof Theorem~\ref{thm:expo_bound} is more classical and follows the same line as the proof of the exponential asymptotic optimality of \cite{GGY23,HXCW24}. The first ingredient is Lemma~\ref{lem:locally_linear} that
 shows that, by non-degeneracy, there exists a neighborhood of $\bx^*$ such that $\mu_\tau(\bx^*)$ is locally linear around $\bx^*$. The second ingredient is Lemma~\ref{lem:concentration}
 that shows that $\bX{t}$ is concentrated around $\bx^*$ as $t\to\infty$, \emph{i.e.}, for all $\epsilon>0$, there exists $C>0$ such that $\lim_{t\to\infty} \Pr[\|\bX{t}-\bx^*\|\ge \varepsilon]\le e^{-C N}$. Combining the two points imply the result. For more details, see 
 Appendix~\ref{apx:PFcomp_expo}.

\begin{figure}[b]
    \centering
    \begin{tabular}{@{}c@{}c@{}c@{}}
        \includegraphics[width=0.33\linewidth]{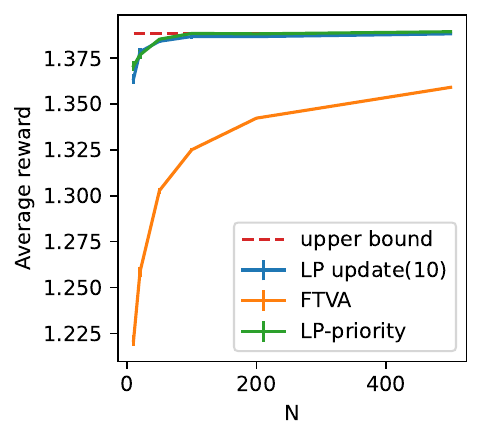}
        &\includegraphics[width=0.33\linewidth]{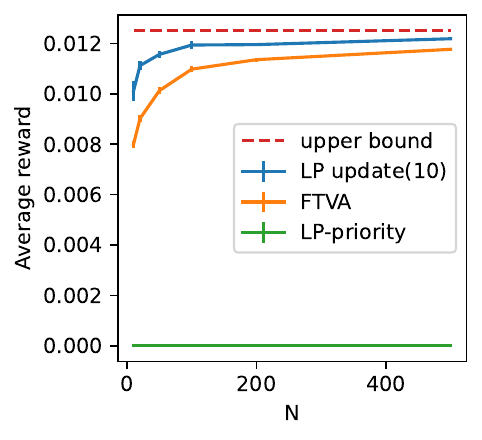}
        &\includegraphics[width=0.33\linewidth]{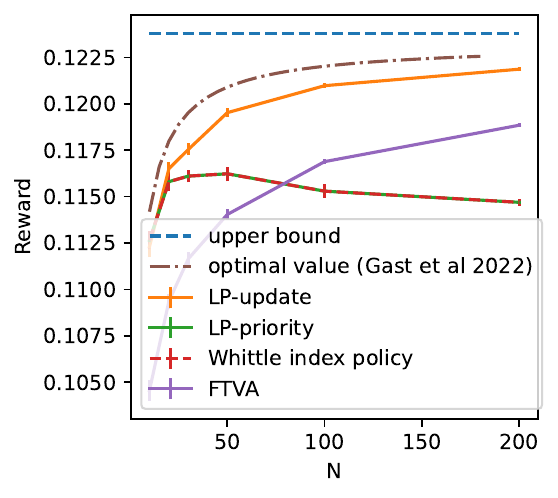}\\
        (a) Random example
        &(b) Example \cite{HXCW23}
        &(c) Example \cite{chen:tel-04068056}
    \end{tabular}
    \caption{Performance as a function of $N$}
    \label{fig:perf_function_of_N}
\end{figure}

\section{Numerical illustrations}
\label{sec:numerical}

We illustrate numerically the performance of the LP-update policy, showing that it performs very well in practice and outperforms classical heuristics in most cases. We choose to compare against two heuristics of the literature: the LP-priority policy of \cite{GGY23} and the follow the virtual advice policy (FTVA) of \cite{HXCW23}. The LP-priority policy computes a solution to \eqref{EQ::VOPTDYN-Tinf} and uses it to compute a priority order on states, arms are then pulled according to this priority order.  The LP-priority policy is an index policy ---similarly to Whittle index policy--- but has the advantage of being always defined, whereas Whittle index policy is only defined under a technical condition known as indexability. The LP-update policy has only been shown to be optimal under the UGAP assumption. 
FTVA constructs a virtual collection of $N$ arms which follow the solution to \eqref{EQ::VOPTDYN-Tinf}. When a virtual arm is pulled, the corresponding real arm is pulled if the budget allows it, otherwise the real arm is left alone. If the virtual and real arm choose the same actions, they evolve identically, otherwise independently. We choose these two heuristics because they are natural and simple to implement and do not rely on hard-to-tune hyperparameters (further discussion for our choice can be found in the supplementary material along with choices for the parameters, see Appendix ~\ref{apx:parameters}).
We provide the parameters in Appendix~\ref{apx:parameters}.All codes are provided in supplementary material.

\paragraph{Comparison on representative examples} For our first comparison, we consider three representative examples: a randomly generated example (in dimension 8), the main example used in \cite{HXCW23} and described in their Appendix~G.2 (abbreviated Example~\cite{HXCW23} in the following), and Example~2 of Figure~7.4 of \cite{chen:tel-04068056} (abbreviated Example~\cite{chen:tel-04068056}). We report the result in Figure~\ref{fig:perf_function_of_N}. We observe that in all three cases, both the LP-update and the FTVA policies are asymptotically optimal, but the LP-update policy outperforms FTVA substantially. The situation for LP-priority is quite different: as shown in \cite{GGY23}, in dimension 8, the LP-priority policy is asymptotically optimal for most of the randomly generated examples. This is the case for example (a) of Figure~\ref{fig:perf_function_of_N} for which LP-update and LP-priority give an essentially equivalent performance and are essentially almost optimal. The authors of \cite{chen:tel-04068056} provide the numerical value of the optimal policy for Example~\cite{chen:tel-04068056}, which can be done because $|\sspace|=3$. 
The LP-update is very close to this value.

\paragraph{System dynamics:} To explore the difference of performance between the LP-update policy and FTVA, we study the dynamics of the different policies for the first two examples of Figure~\ref{fig:perf_function_of_N}. 
Figure~\ref{fig:as_function_of_time} 
plots the evolution of the system over time (for a single trajectory). We observe that if the distance between $X_t$ and $x^*$ are similar for LP-update and FTVA, the rotated cost is much smaller for LP-update, which explains its good performance. 
\begin{figure}[ht]
    \centering
    \begin{tabular}{@{}c@{}c@{}@{}c@{}c}
        \includegraphics[width=0.25\linewidth]{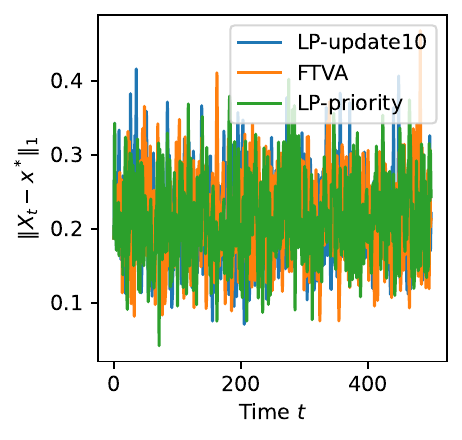}
        &\includegraphics[width=0.25\linewidth]{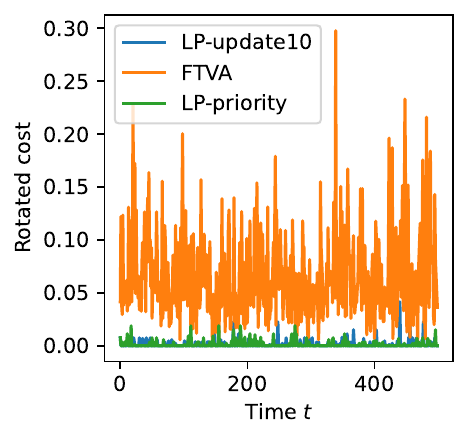}
        &\includegraphics[width=0.25\linewidth]{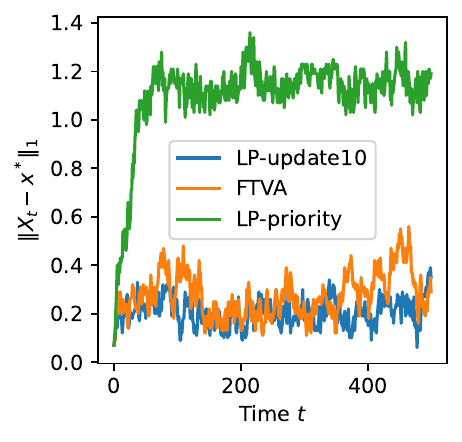}
        &\includegraphics[width=0.25\linewidth]{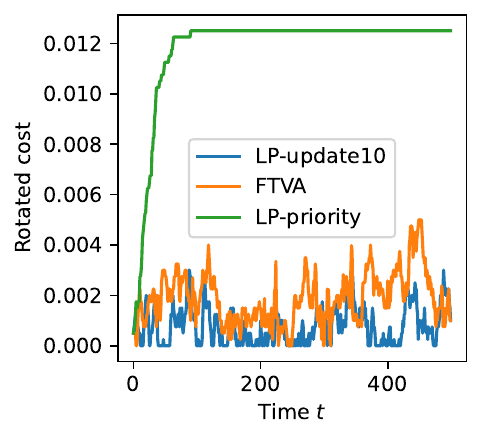}
        \\
        \multicolumn{2}{c}{(a) Random example}
        &\multicolumn{2}{c}{(b) Example from Hong et al.}
    \end{tabular}

    \caption{Distance $\|X_t-x^*\|$ and rotated cost as a function of time for $N=100$.}
    \label{fig:as_function_of_time}
\end{figure}
To explore the system dynamics in more details, we focus on Example~\cite{chen:tel-04068056} and study the behavior of the LP-update, FTVA and LP-priority policy. In Figures~\ref{fig:example-chen}(a,b,c), we present a trajectory of each of the three policies: each orange point corresponds to a value of $X_t\in\Delta_\sspace$ (as this example is in dimension $|\sspace|=3$, the simplex $\Delta_\sspace$ can be represented as a triangle). This example has an unstable fixed point (Assumption~\ref{AS::stable} is not satisfied). Both LP-update and FTVA concentrate their behavior around this fixed point but the LP-priority policy exhibits two modes. 
The rotated cost (Figure~\ref{fig:example-chen}(d)) 
is much smaller for LP-update when compared to FTVA. This explains why the LP-update performs better as shown in Figure~\ref{fig:perf_function_of_N}(c).
\begin{figure}[ht]
    \centering
    \begin{tabular}{@{}c@{}c@{}c@{}c@{}}
        \includegraphics[width=.25\linewidth]{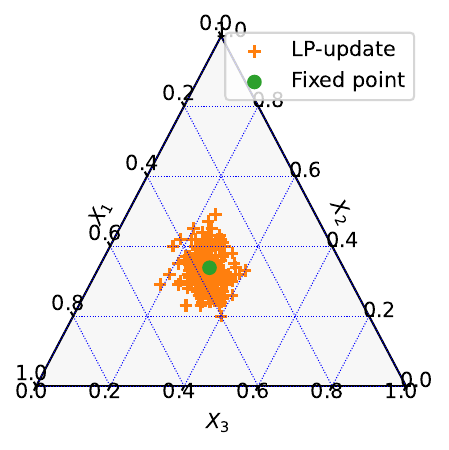}
        &\includegraphics[width=.25\linewidth]{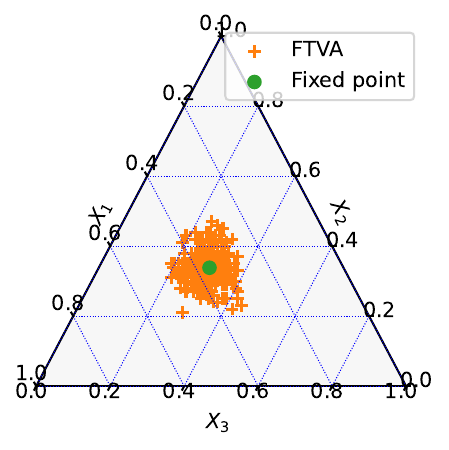}
        &\includegraphics[width=.25\linewidth]{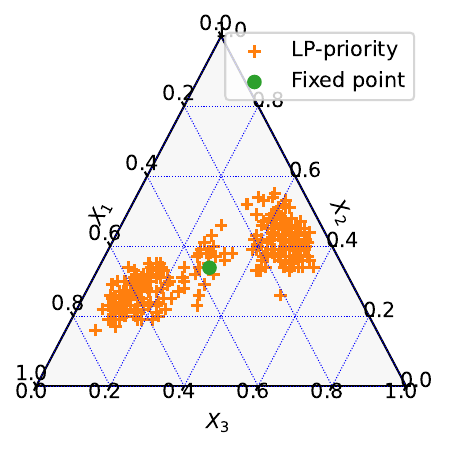}
        &\includegraphics[width=.25\linewidth]{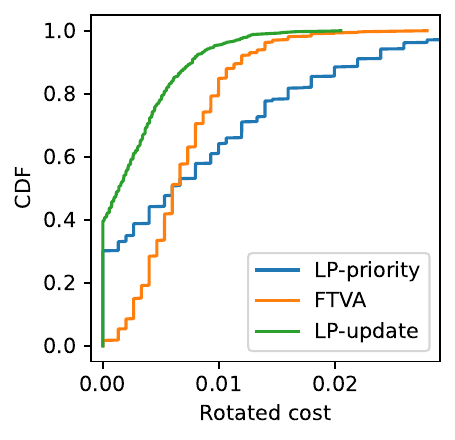}\\
        (a) LP-update
        &(b) FTVA
        &(c) LP-priority        
        &(d) CDF of rotated cost
    \end{tabular}

    \caption{Example~2 from Gast et al. 22. Simulation for $N=100$.}

    \label{fig:example-chen}
\end{figure}

\paragraph{Influence of parameters} Figure~\ref{fig:gain_function_parameters}, studies the influence of different parameters on the performance of LP-update. In each case, we generated $20$ examples and plot the average ``normalized'' performance among these examples. 
``normalized'' here means we divide the performance of the policy by the value of the LP problem \eqref{EQ::VOPTDYN-Tinf}. The first plot~\ref{fig:gain_function_parameters}(a) shows that the influence of the parameter $\tau$ is marginal (the curves are not distinguishable for $\tau\in\{3,5,10\}$).  Plot \ref{fig:gain_function_parameters}(b) indicates that the sub-optimality gap of LP-update is not too sensitive to the state space size unlike FTVA which degrades when $|\sspace|$ is large probably because there are fewer synchronized arms. The last plot \ref{fig:gain_function_parameters}(c) studies the performance as a function of the budget parameter, $\alpha$. 
\begin{figure}[t]
    \centering
    \begin{tabular}{@{}c@{}c@{}c@{}}
        \includegraphics[width=0.33\linewidth]{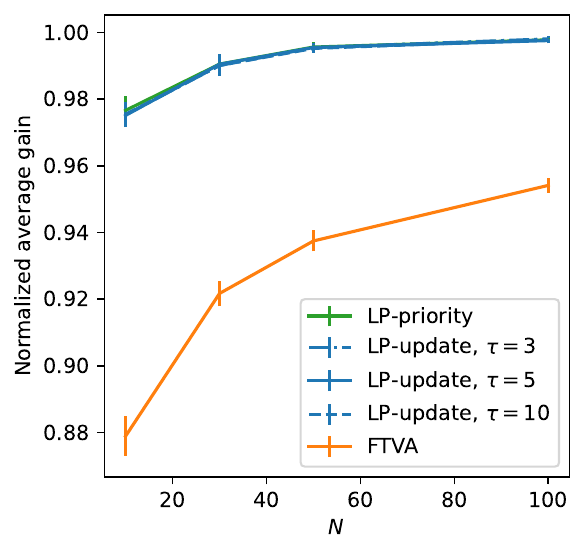}
        &\includegraphics[width=0.33\linewidth]{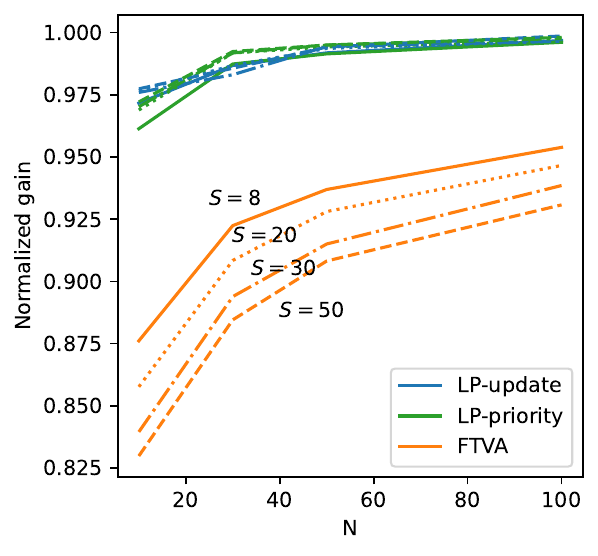}
        &\includegraphics[width=0.33\linewidth]{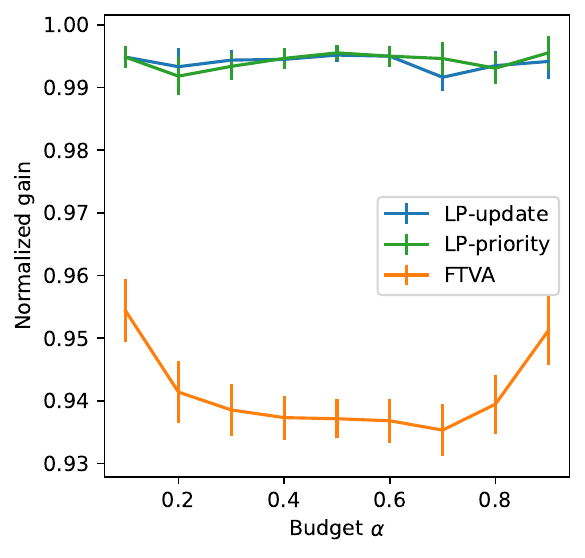}\\
        (a) For various $\tau$.
        &(b) For various $S$.
        &(c) For various $\alpha$ ($N=50$).
    \end{tabular}

    \caption{Comparison of the gains as a function of some parameters.}
    \label{fig:gain_function_parameters}
\end{figure}
\section{Conclusion}
\label{sec:conclusion}
In this paper, we study the problem of constructing an efficient policy for restless multi-armed bandit for the average reward case. We show that under a quite general condition, a simple model-predictive control algorithm provides a solution that achieves both the best known upper bound ($O(1/\sqrt{N})$ or $\exp(-\Omega(N))$ for stable non-degenerate problems), and also works very efficiently in practice. Our paper provides the first analysis of this policy for the average reward criterion. By using a novel framework based on dissipativity we are able to make a subtle connection between the finite and the infinite-horizon problems and we are the first to use it in this context. We believe that this framework is what makes our analysis simple and easily generalizable to other settings. As an example, we discuss in 
Appendix~\ref{apx:generalization}  
the generalization to multi-action multi-constraint bandits. This connection also lends itself to a model-based learning setting. Furthermore, we believe that this framework could be used in the heterogeneous bandit problem. These problems are a matter for future work. 

\bibliography{COLT_2025_conference}
\newpage
\appendix

\section{Additional Literature Review}\label{apx:review}
The study of bandit problems dates back at least to the 1930's in the form of sequential design of experiments, \citep{Thompson33}, \citep{Wald:1947}. However, the more modern perspective of Markovian bandits in terms of MDPs is largely credited to the monograph by \citet{bellman1957dynamic}. This monograph brought to light the combinatorial difficulty of this problem even over a finite horizon. The \emph{rested} variant of the multi-armed bandit problem was famously resolved through \citet{Gittin79} which proposed a remarkably simple structure to the solution by allotting indices to the states of the bandit arms.

In his seminal work \citet{Wh88} generalized the problem to the \emph{restless multi-armed bandit problem} in its modern form and conjectured that under a condition known as indexability, an index policy is asymptotically optimal (in the number of arms) for the RMAB problem. \citet{WW90} showed a counter example to this conjecture, further going on to show in the same work that the conjecture did hold under a \emph{global attractor} (UGAP assumption) and ergodicity condition of a single-arm. This result was generalized to the multi-action setting by \citet{Glazebrook15}. On the other hand, \citet{verloop2016asymptotically, GGY23b} generalized the set of index policies by constructing a set of priority policies known as LP priority index while maintaining the UGAP assumption. More surprisingly, \citet{GGY23, GGY23b} showed these priority policies are in fact exponentially close to being optimal, under three conditions: uniform global attractor property (UGAP), \emph{non-degeneracy} of the fixed point and global exponential stability \cite{GGY23}. The restriction of indexability is removed in \cite{GGY23b} and the condition of UGAP was first removed \cite{HXCW24}.

\citet{GGY23, GGY23b} observed that under the conditions of non-degeneracy, there exists a local neighborhood around the fixed point where the priority policy set is affine. This is the key observation which leads to the exponentially small error bounds under simple local stability conditions. This observation around the fixed point drove several works \citep{HXCW23, HXCW24,hong2024unichain, Avrachenkov2024, yan2024} to shift their perspective towards policies that look to steer the dynamical system towards the optimal fixed point. \citet{HXCW23} created a virtual system that was driven to the fixed point. To transform this into an asymptotically optimal policy, they required the real system to synchronize with the virtual system, hence, they developed the \emph{synchronization assumption}. Unlike \citet{HXCW23} which passively allowed arms to align with each other eventually, \citet{HXCW24} used a \emph{two set policy} to actively align non aligned arms to the fixed point. In contrast to this form of control, \citet{yan2024} designed the \emph{align and steer} policy to steer the mean field dynamical control system towards the fixed point. \citet{yan2024}'s policy is asymptotically optimal under the assumptions of \emph{controllability} of the dynamical system to more directly move the state of the system towards the fixed point over time. In parallel to this work, \citet{Avrachenkov2024} used a similar idea of mean field dynamical control for weakly coupled systems under a \emph{relaxed but generalized constraint set} if the single arm process is unichain and aperiodic.

In contrast to these approaches our algorithm is not explicitly designed at the outset to steer the corresponding dynamical system towards a fixed point. Rather, we show that the model predictive control algorithm is designed to solve an equivalent \emph{rotated cost minimization} problem. In doing so, it produces policies that are close to optimal to the infinite horizon average reward problem. We  thus produce a policy that is easily generalizable to the multi-action, general constraint setting \emph{without relaxing the constraints} using ideas from Whittle's relaxation. Further, when the fixed point is unique, minimizing the rotated cost problem will coincide with steering the dynamical control system towards the fixed point, allowing us to recover exponentially close asymptotic optimality bounds.

From an algorithmic perspective, there are also a lot of recent papers on RMABs for the finite-horizon reward criteria 
\cite{hu2017asymptotically,zayas2019asymptotically,brown2020index,ghosh2022indexability,zhang2021restless,GGY23}, or the infinite-horizon discounted reward \cite{zhang2022near,ghosh2022indexability}. Our results are close in spirit to those since our policy shows how to use a finite-horizon policy to construct a policy that is asymptotically optimal for the infinite-horizon rewards. 

Finally, it is worth noting that a recent result \cite{yan2024optimalgap}, developed in parallel to our own showed an order $\mcl{O}(1/N)$ result for the \emph{degenerate finite time horizon problem} using a diffusion approximation instead of a mean field approximation. Unfortunately, this finite horizon result cannot immediately be compared to the infinite horizon average reward result primarily because unlike the mean field solution which can directly be translated as the solution to \eqref{EQ::VOPTDYN-Tinf} the diffusion dynamic has no immediate corresponding equivalent.

\section{Algorithm details}
\label{apx:algo}

\subsection{Rounding procedure}
\label{apx:rounding}

The rounding procedure is composed of two steps:
\begin{enumerate}
    \item First, given $\bu = \mu_\tau(\bx(\bS(t)))$, we need to construct a vector $\bU^N$ that is \emph{as close as possible} to $\bu$ while satisfying the constraints that $\sum_i u_i \leq \alpha$ and $N\bU^N$ is an integer. 
    \item Second, we use $\bU^N$ to construct a feasible sequence of actions $\bA(t)\in\{0,1\}^N$ that we can apply to $\bS(t)$. 
\end{enumerate}
The second step is quite easy: as $Nx_i$ and $NU^N_i$ are integers for all $i$, we can construct a feasible sequence of action by applying the action $1$ to $NU^N_i$ arms that in state $i$ and action $0$ to the $N(x_i-U^N_i)$ others. The first step is more complicated because we want to construct a $\bU^i$ that is \emph{as close as possible} as $u_i$. Before giving our procedure, let us illustrate it through two examples of $(N\bx,N\bu)$ for $\alpha=0.5$:
\begin{center}
    \begin{tabular}{|c|c||l|r|}
        \hline
         $N\bx$ & $N\bu$  & Difficulty & Possible $\bU^N$ by Algorithm~\ref{algo::rounding}\\\hline
         $(10,10,10,9)$&$(10,9.5,0,0)$& $\|u\|_1 \ge \alpha$& $(10, 9, 0, 0)$ \\\hline
         && We need to randomize& $(10,6,0,0)$ with proba $0.7$\\
         $(10,10, 10, 9)$ &$(10, 5.7, 0.2, 0)$&to be as close as& $(10,5,1,0)$ with proba $0.2$\\
         &&possible to $\bu$&$(10,5,0,0)$ with proba $0.1$\\\hline
         $(10,10, 10, 9)$ &$(10, 4.9, 4.6, 0)$&$\|u\|_1\ge\alpha$ and we& $(10,5,4,0)$ with proba $0.4$\\
         &&need to randomize &$(10,4,5,0)$ with proba $0.6$\\\hline
    \end{tabular}
\end{center}

This leads us to construct Algorithm~\ref{algo::rounding}. This algorithm first construct a $\bv$ that is as close as possible to $\bu$ while satisfying $\|\bv\|\le\alpha$. There might be multiple choices for this first step and the algorithm can choose any.  For instance, applied to the example such that $\bx=(10,10, 10, 9)$ and $\bu=(10, 4.9, 4.6, 0)$, this algorithm would first construct a vector $\bv$ that can be any convex combination of $(10, 4.4, 4.6, 0, 0)$ and $(10, 4.9, 4.1, 0, 0)$. For instance, if the algorithm chose the vector $\bv=(10, 4.4, 4.6, 0)$, then the algorithm would produce $(10,5,4,0)$ with probability $0.4$ and $(10,4,5,0)$ with probability $0.6$.  Once this is done, the algorithm outputs a $\bU^N$ that satisfies the constraints and such that $\E[\bU^N]=\bv$. In some cases, there might be multiple distributions that work. The algorithm may output any. An efficient procedure to implemented this last step is provided in Section 5.2.3 of \cite{ioannidis2016adaptive}. 

\begin{algorithm}[ht]
	\caption{Randomized rounding}
	\label{algo::rounding}
	\begin{algorithmic}
        \REQUIRE Integer $N$, vector $\bx\in\Delta_\sspace$ such that $Nx_i$ is an integer for all $i$, vector $\bu\in\mathcal{U}(\bx)$.
        \STATE Let $\bv$ be (any) vector such that $\bv\le\bu$ and $\|\bv\|_1 = \min\left(\|\bu\|_1,\frac{\lfloor \alpha N\rfloor}{N}\right)$.
        \STATE For all state $i$, let $z_i := Nv_i - \lfloor N v_i \rfloor$ be the fractional part of $Nv_i$. 
        \STATE Sample a sequence of $|\sspace|$ Bernoulli random variables $\bZ=(Z_1\dots \bZ_{|\sspace|})$ such that $\E[Z_i]=z_i$ and $\sum_{i=1}Z_i\le \lceil \sum_{i=1}z_i\rceil$.
        \ENSURE $(\floor{Nv_i} + Z_i)_{i\in\sspace}$.
    \end{algorithmic}
\end{algorithm}

\begin{lemma}\label{apx::lem_round}
    Algorithm~\ref{algo::rounding} outputs a random vector $\bU^N$ such that $\bU^N\in\mathcal{U}(\bx)$ and such that
    \begin{align*}
        \|\E[\bU^N] - \bu \|_1 \le \frac{\floor{\alpha N}-\alpha N}{N}.
    \end{align*}
\end{lemma}

\subsection{Generalization to multi-constraints MDPs}
\label{apx:generalization}

In this section, we give some clues on how and why our algorithm and results can be generalized to multi-action multi-constraints MDPs. Specifically, we make the following modifications to our model:
\begin{itemize}
    \item Instead of restricting an action to be in $\{0,1\}$, we can consider any finite action set $\mathcal{A} = \mathcal{A} =  \{0\dots A-1\}$.
    \item Instead of having a single constraint $\sum_n A_n(t)\le \alpha N$, we consider that there are $K$ types of constraints and that when taking action $a$ in state $i$, this action consumes $D^a_{i,k}\ge0$ of resource $k$. The action $0$ is special and does not consume resources. We impose $K$ resource constraints:
    \begin{align*}
        \sum_{n=1}^N D^{A_n(t)}_{S_n(t), k} \le C_k && \forall k\in{1\dots K}.
    \end{align*}
\end{itemize}
As these new constraints are linear, one can define LP relaxation equivalence of \eqref{EQ::VOPTDYN-Tinf} and \eqref{EQ::VOPTDYN}, which leads to an LP-update adapted to this new setting. The only major modification of the algorithm concerns the randomized rounding and is where one would need to have an action $0$ in order to guarantee the feasibility of the solution. 

\begin{itemize}
    \item The generalization of Assumption~\ref{AS::SA} to the multi-action case is straightforward. Hence, one can prove an equivalent of Theorem~\ref{thm:asymptotic_optimal} for multi-action multi-constrained MDPs with a rate of convergence of $O(1/\sqrt{N})$. 
    \item To obtain a generalization of Theorem~\ref{thm:expo_bound}, there are two main difficulties
    \begin{enumerate}
        \item The first is to redefine the notion of a non-degenerate fixed-point in order to replace Assumption~\ref{AS::non-degenerate} by one adapted to the multi-action multi-constraint case. To that end, we believe that the most appropriate notion is the one presented in \cite{gast2024reoptimization}. It provides the notion of non-degeneracy that uses a linear map.
        \item The second is to provide an equivalent of the stability Assumption~\ref{AS::stable}. To do so, we can again use the notion of non-degeneracy defined in \cite{gast2024reoptimization} that defines a linear map. The stability of this linear map should suffice to prove the theorem.
    \end{enumerate}
\end{itemize}

\section{Proof of Theorem~\ref{thm:asymptotic_optimal}}\label{apx:PFcomp}
\subsection{Part 1, Properties of the dynamical control problem \ref{EQ::VOPTDYN}:}
We begin by formally defining the space of stationary policies $\pspacedyn$ as the set of all policies that map $\bx \in \Dels$ to an action $\bu \in \mcl{U}(\bx)$. Note, any policy in $\pspace$ must satisfy this condition. Now, for an initial distribution $\bx \in \Dels$ define the discounted infinite horizon reward problem,
\begin{align}\label{EQ::VDISC}
    \Vdetdisc{\bx}{} :=& \max_{\pi \in \pspacedyn}\sum_{t = 0}^{\infty} \beta^{t}R(\bx(t), \bu(t))\\
    \text{Subject to} & \text{ the dynamics \eqref{EQ:MeanField}}
\end{align}
where $x(t)$ denotes the trajectory induced by the policy $\pi$. Note, this limit is always well defined for $\beta < 1$ and the optimal policy exists \citep{puterman2014markov}. The main idea of this section will be an exercise in taking appropriate limit sequences of the discount factor $\beta$ to define the gain and bias for the constrained average reward problem.

Before we proceed we introduce a little notation for convenience in writing down our proof. For any pair $(\bx, \bu)$  define an equivalent $\{\y{s}{a}\}_{s, a}$ with:
\begin{align}\label{EQ::StateAction}
    \y{s}{1} &:= \bu_{s}\\
    \y{s}{0} &:= \bx_{s} - \bu_{s}
\end{align}
Here $\y{s}{a}(t)$ represents the fraction of arms in state $s$, taking action $a$ at time $t$. It is also convenient to introduce a concatenated reward vector $\Rvec{} := [\Rvec{0}, \Rvec{1}]$ and a concatenated transition kernel $\Pnext := [\Pmat{0}, \Pmat{1}]$. If $\yb := \{\y{s}{a}\}$ represents the $\sspace \times 2$ vector, it is not hard to check that $R(\bx, \bu) := \Rvec{0}\cdot \bx + (\Rvec{1} - \Rvec{0}) \cdot\bu = \Rvec{} \cdot \yb$ and $\Phi(\bx, \bu) := \Pmat{0} \cdot \bx + (\Pmat{1} - \Pmat{0}) \cdot \bu = \Pnext \cdot \yb$. We can rewrite the discounted problem as follows:
\begin{align}
    \Vdetdisc{\bx}{} :=& \max_{\pi \in \pspacedyn}\sum_{t = 0}^{\infty} \beta^{t} \Rvec{} \cdot \yb (t) \label{EQ:Vdis}\\
    \y{s}{0}(t) +  \y{s}{1}(t) =& \bx_{s}(t) \label{EQ:DYN1}\\
   \sum_{a} \y{s}{a}(t + 1) =& \sum_{s', a'} P^{a'}_{s', s} \y{s'}{a'}(t) \hspace{0.2 in} \forall \hspace{0.05 in} t \label{EQ:DYN2}\\
   \text{s.t.}\hspace{0.4 in} & \nonumber\\
   \sum_{s} \y{s}{1}(t) \leq& \alpha\hspace{0.2 in} \forall \hspace{0.05 in} t \label{EQ:DYN3}
\end{align}
We now state the following lemma on the continuity of the value function in the state:  
\begin{lemma}\label{LEM::LLIP}
    Under assumption \ref{AS::SA}, for any $\beta \in (0, 1]$ we have,
    \begin{equation}
        \Vdetdisc{\bx}{} - \Vdetdisc{\bx'}{} \leq \frac{k}{\syncconst_k}\|\bx - \bx'\|_1
    \end{equation}
\end{lemma}
Note, this continuity result is independent of the discount factor $\beta$. We postpone the proof to Appendix \ref{apx::section_Lip}. This will be critical to proving the main result of the subsection below.
\begin{lemma}\label{LEM::EXGH}
    Consider the infinite horizon average reward problem under the synchronization assumption. There exists a constant gain and a bias function from $\Dels \to \R$ denoted by $g^{\star}$ and $h^{\star}(\cdot)$ respectively defined by:
    \begin{equation*}
        g^{\star} := \lim_{i \to \infty} (1 - \beta_i) \Vdetdisc{\bx}{\beta_i}
    \end{equation*}
    for an appropriate sequence $\beta_i \to 1$ as $i \to \infty$. For the same sequence of $\beta_i$ we can define the Lipschitz continuous bias function,
    \begin{equation*}
        h^{\star}(\mB) := \lim_{i \to \infty} \sum_{t = 0}^{\infty} \beta_i^{t} R(\bx(t), \bu(t)) - (1 - \beta_i)^{-1}g^{\star}
    \end{equation*}
     Furthermore, they satisfy the following fixed point equations,
    \begin{equation}\label{eq::disc1}
        g^{\star} + h^{\star}(\mB) = \max_{\bu \in \mcl{U}(\bx)} R(\bx(t), \bu(t)) +   h^{\star}(\Phi(\bx, \bu))
    \end{equation}
\end{lemma}
\begin{proof}
Now note, $\yb \cdot \one = 1$, where $\one$ is the all $1'$s vector. We will overload the notation slightly by letting $\yb(t)$ be the optimal trajectory taken by the optimal policy $\pi^{*}_{\beta}$ for \ref{EQ::VDISC}. Hence, for any constant $g$ we have:
\begin{equation}\label{eq::disc2}
    \Vdetdisc{\bx}{} = (1 - \beta)^{-1} g + \sum_{t = 0}^{\infty} \beta^{t} [\Rvec{} - g \one] \cdot \yb (t)
\end{equation}
One can check that $\Vdetdisc{\bx}{}$ satisfies the following \emph{Bellman equation},
\begin{equation}\label{eq::disc3}
   \Vdetdisc{\bx}{} = \Rvec{}\cdot \yb(t) + \beta \Vdetdisc{\Pnext \cdot \yb}{}
\end{equation}
Now note, given $\yb(t)$ at time $t$, $\bx(t + 1)$ is given by a linear (hence, continuous) map $\bx(t + 1) := \Pnext \cdot \yb(t)$. Further note, given $\pi^{*}_{\beta}$, $\yb(t)$ is upper hemicontinuous with respect to $\bx(t)$. Let $\mH_{\beta}$ denote the map induced by $\pi^*_{\beta}$ from  $\bx$ to $\yb$, then $\mH_{\beta} (\bx(t)) = \yb(t)$. Combining the two maps, we have $\Pnext \mH_{\beta} : \Dels \to \Dels$. Since, $\Dels$ is closed and bounded, by Kakutani's fixed point theorem there exists atleast one fixed point $\bx_{\beta} := \Pnext \mH_{\beta} \cdot \bx_{\beta}$ if $\yb_{\beta}$ is the corresponding value for $\bx_{\beta}$ we have $\bx_{\beta} := \Pnext \cdot \yb_{\beta}$. Combining these observations allows us to write,
\begin{equation*}
    \Vdetdisc{\bx_{\beta}}{} = \Rvec{}\cdot \yb_{\beta} + \beta \Vdetdisc{\bx_{\beta}}{}
\end{equation*}
We will set $g_{\beta} := \Rvec{}\cdot \yb_{\beta}$ and designate this value as the \emph{gain} of our problem. It then follows that $\Vdetdisc{\bx_{\beta}}{} = (1 - \beta)^{-1}g_{\beta}$. Now, we can define a \emph{bias function} $h_{\beta}(\cdot)$ as the remaining terms of the equation \eqref{eq::disc2}, concretely,
\begin{equation}\label{eq::disch}
    h_{\beta}(\bx) = \sum_{t = 0}^{\infty} \beta^{t} [\Rvec{} - g_{\beta} \one] \cdot \yb (t)
\end{equation}
and this definition of $h_{\beta}(\cdot)$ is well defined since the sum is bounded. It now follows that,
\begin{equation}\label{eq::discvalue}
    \Vdetdisc{\bx}{} = (1 - \beta)^{-1} g_{\beta} + h_{\beta}(\bx)
\end{equation}
More importantly, by plugging these definitions into \eqref{eq::disc3} and making the dependence of $\yb$ on the maximal policy explicit, we obtain the following recursion equation: 
\begin{equation}\label{eq::discfp}
    g_{\beta} + h_{\beta}(\bx) = \max_{\yb: \{\yb_{1} \in \mcl{U}(\bx)\}} g_{\beta} + \beta h_{\beta}(\Pnext \cdot \yb)
\end{equation}
Further, due to the Lipschitz continuity, Lemma \ref{LEM::LLIP} of $\Vdetdisc{\bx}{}$ with respect to $\bx$ we have, for any discount factor $\beta_i < 1$,
\begin{align*}
   (1 - \beta_i)\|\Vdetdisc{\bx}{\beta_i}  - \Vdetdisc{\bx'}{\beta_i}\|_1 \leq \frac{k (1 - \beta_i)}{\syncconst_k}\|\bx - \bx'\|_1 
\end{align*}
. Hence, by choosing a sequence $\beta_i \uparrow 1$ as $i \to \infty$ we obtain:
\begin{equation*}
    \lim_{i \to \infty}(1 - \beta_i)\|\Vdetdisc{\bx}{\beta_i}  - \Vdetdisc{\bx'}{\beta_i}\|_1 \to 0
\end{equation*}
. From \citet{puterman2014markov} we know there exists a sequence $i \to \infty$ such that $(1 - \beta_i)\Vdetdisc{\bx}{\beta_i} \to \lim_{\THor \uparrow \infty}\frac{\VoptInf{\THor}{\bx}}{\THor}$. In particular, this implies that $(1 - \beta_i)\Vdetdisc{\bx}{\beta_i}$ converges to a constant value which we shall denote by $g^{\star}$. This gives us the first result. 
Critically, by noting that $g_{\beta_i} := (1 - \beta_i)\Vdetdisc{\bx_{\beta_i}}{\beta_i}$ we see that $g^{\star}$ must be the limit point of $R(\bx_{\beta_i}, \bu_{\beta_i}) =: \Rvec{} \cdot \yb_{\beta_i} =: g_{\beta_i}$, hence, it is the solution to \eqref{EQ::VOPTDYN-Tinf}. 

Next, once again leveraging Lemma \ref{LEM::LLIP} we have, for any $\beta < 1$,
\[
\|\Vdetdisc{\bx}{}  - (1 - \beta)^{-1}g_{\beta}\|_1 :=\|\Vdetdisc{\bx}{}  - \Vdetdisc{\bx_{\beta}}{}\|_1 \leq \frac{k}{\syncconst_k}\|\bx - \bx_{\beta}\|_1
\]
Plugging these results into \eqref{eq::discvalue} we see that,
\begin{equation}\label{eq::disc_lip_h}
    \|h_{\beta}(\bx)\|_1 \leq \frac{k}{\syncconst_k}\|\bx - \bx_{\beta}\|_1
\end{equation}
is bounded by a constant for all $\beta < 1$. By the same token $h_{\beta}$ is Lipschitz with constant $\frac{k}{\syncconst_k}$. Choosing the same sequence $\beta_i$ in \eqref{eq::disch} and passing to the limit we obtain :
\begin{equation*}
    h^{\star}(\bx) := \lim_{i \to \infty}h_{\beta_i}(\bx) = \lim_{i \to \infty}\sum_{t = 0}^{\infty} \beta_{i}^{t} [\Rvec{} - g_{\beta_i} \one] \cdot \yb (t) = \lim_{i \to \infty}\sum_{t = 0}^{\infty} \beta_{i}^{t} \Rvec{}\cdot \yb(t) - (1 - \beta_i)^{-1}g^{\star}
\end{equation*}
 Thus, we have defined both the gain and bias $g^{\star}$ and $h^{\star}(\cdot)$ for our \emph{deterministic average reward problem}. Recall, $R(\bx, \bu) := \Rvec{0} \cdot \bx + (\Rvec{1} - \Rvec{0}) \cdot \bu = \Rvec{} \cdot \yb$ giving us the second result: 
 \begin{equation}
         h^{\star}(\bx) = \lim_{i \to \infty}\sum_{t = 0}^{\infty} \beta_{i}^{t} R(\bx(t), \bu(t)) - (1 - \beta_i)^{-1}g^{\star}
 \end{equation}
 Further, this sequence $\beta_i \uparrow 1$ in \eqref{eq::discfp} yields the following fixed point equation:
\[
 g^{\star} + h^{\star}(\mB) = \max_{\bu \in \mcl{U}(\bx)} R(\bx(t), \bu(t)) +   h^{\star}(\Phi(\bx(t), \bu(t)))
\]
which completes the proof.
\end{proof}

\subsection{Part~2: Dissipativity}
\label{apx:dissipativity}

\begin{lemma}
    \label{LEM::DISS}   
    Let $\rotcost{l}{\bx}{\bu} = g^* - R(\bx,\bu) + \lambda \cdot \bx - \lambda \cdot \Phi(\bx,\bu)$. Then:
    \begin{itemize}
        \item $\rotcost{l}{\bx}{\bu}\ge0$ for all $\bx\in\Delta_{\sspace}$ and $\bu\in\mathcal{U}(\bx)$.
        \item If $(\bx^*,\bu^*)$ is a solution to \eqref{EQ::VOPTDYN-Tinf}, then $\rotcost{l}{\bx^*}{\bu^*}=0$.
    \end{itemize}
    This shows that our problem is dissipative.
\end{lemma}
\begin{proof}
    Recall that $\lambda$ is the optimal dual variable of the constraint $\mB = \Pmat{0} \cdot \mB +  (\Pmat{1} - \Pmat{0}) \cdot \bu$ of problem \eqref{EQ::VOPTDYN-Tinf}. By strong duality and Lemma ~\ref{LEM::EXGH}, this implies that
    \begin{align*}
       g^* &= \max_{\bx, \bu}~ R(\bx, \bu) + \lambda \cdot (\Pmat{0} \cdot \bx +  (\Pmat{1} - \Pmat{0}) \cdot \bu)-\lambda \cdot \bx 
        &\text{Subject to: }\quad
        \bu \in \mcl{U}(\mB).
    \end{align*}
    Recall that $\Phi(\bx,\bu)=\Pmat{0} \cdot \bx +  (\Pmat{1} - \Pmat{0}) \cdot \bu$. This implies that
    \begin{align*}
       0 &= \min_{\bx, \bu}~ \rotcost{l}{\bx}{\bu}
        &\text{Subject to: }
        \bu \in \mcl{U}(\bx),
    \end{align*}
    which implies the results of the lemma.
\end{proof}
    
\subsection{Part~3: MPC is optimal for the deterministic problem}
\label{apx:MPC}

Recall that $\Costmin{\tau}{\mB} := \min \sum_{t = 0}^{\tau - 1} \rotcost{l}{\bx(t)}{\bu(t)}$ with $\bx(0) = \mB$.
\begin{lemma}
    \label{lem:C_vs_W}
    Given any initial distribution $\bx \in\Delta_{\sspace}$, the cost minimization problem $\Costmin{\tau}{\bx}$ and the reward maximization problem $\VoptInf{\tau}{\bx}$ are equivalent. Next, for all $\bx$, the limit $\lim_{t \to \infty}\Costmin{t}{\bx}$ exists and is bounded. Further, for any $\epsilon > 0$, there exists $\tau(\epsilon) < \infty$ such that: 
    \[
    |\lim_{t \to \infty}\Costmin{t}{\bx} - \Costmin{\tau(\epsilon)}{\bx}| < \epsilon.
    \]
    Finally, the limit so defined is Lipschitz with Lipschitz constant $C_{L}$ bounded above by $\|\lambda\|_{\infty} + \frac{k}{\syncconst_k}$.  
\end{lemma}
\begin{proof}
Recall that the objective function of the optimization problems $W_\tau$ and $L_\tau$ are:
\begin{align*}
    &\min_{\bu(t) \text{ s.t. } \eqref{EQ::POL}}\sum_{t=0}^{T-1} \rotcost{l}{\bx(t)}{\bu(t)} & \text{ For $L_\tau$}\\
    &\max_{\bu(t) \text{ s.t. } \eqref{EQ::POL}}\sum_{t=0}^{T-1} R(\bx(t),\bu(t)) + \lambda(\bx(\tau)). & \text{ For $W_\tau$}
\end{align*}
By Lemma ~\ref{LEM::DISS}, the rotated cost is $\rotcost{l}{\mB}{\bu} = g^* - R(\bx,\bu)+\lambda(\bx) - \lambda(\Phi(\bx,\bu))$. As any valid control of $L_\tau$ and $W_\tau$ satisfy $\bx(t+1)=\Phi(\bx(t),\bu(t))$, the objective of $L_\tau$ can be rewritten as: 
\begin{align*}
    \sum_{t=0}^{\tau-1} \rotcost{l}{\bx(t)}{\bu(t)} &= \sum_{t=0}^{\tau-1} g^* - R(\bx(t),\bu(t)) + \lambda(\bx(t)) - \lambda(\Phi(\bx(t),\bu(t)))\\
    &= \lambda(\bx(0)) + \tau g^* -  \sum_{t=0}^{T-1} R(\bx(t),\bu(t)) - \lambda(\bx(\tau)),
\end{align*}
As the last two terms correspond to the objective function for $W_\tau$, this shows that 
\begin{equation}\label{EQ:COST_DEF}
    \Costmin{\tau}{\bx}= \tau g^{\star} + \lambda \cdot \bx - \VoptInf{\tau}{\bx}.
\end{equation} 

Clearly, the two objectives are equivalent and $\LPpol{\tau}{\bx} = \bu(0)$. Now note, by Lemma ~\ref{LEM::EXGH} we have, $$\lim_{\tau \to \infty}\Costmin{\tau}{\bx} = \lambda \cdot \bx + h^{\star}(\bx) < \infty$$ is bounded. Due to dissipativity, $\Costmin{\tau}{\mB}$ is monotone increasing in $\tau$. 
Recall, we denote the limit as $\tau \to \infty$ of the cost minimization problem by $\Costmin{\infty}{\bx}$. It follows that there exists a $\tau(\epsilon)$ such that,
\[
|\Costmin{\infty}{\bx} - \Costmin{\tau(\epsilon)}{\bx}| < \epsilon
\].
Finally, note that: 
$$\Costmin{\infty}{\bx} = \lambda \cdot \bx + h^{\star}(\bx). $$
Hence, the Lipschitz constant under the $l1$ norm can explicitly be bounded as follows,
\begin{align*}
    |\Costmin{\infty}{\bx} - \Costmin{\infty}{\bx'}| &\leq |\lambda \cdot (\bx - \bx')| + |h^{\star}(\bx) - h^{\star}(\bx')|\\
    &\leq \|\lambda\|_{\infty}\|\bx - \bx'\|_1 + \frac{k}{\syncconst_k}\|\bx - \bx'\|_1\\
    &= \left(\|\lambda\|_{\infty} + \frac{k}{\syncconst_k}\right)\|\bx - \bx'\|_1
\end{align*}
Where we used Holder's inequality and Lemma \ref{LEM::EXGH} for the second inequality.  
\end{proof}

\subsection{Proof of Theorem~\ref{thm:asymptotic_optimal}: Computation details for \eqref{EQ:EQV4}}
\label{apx:proof_thm1_details}

Here we detail the analysis of $(A)$ and how we go from \eqref{EQ:EQV3} to \eqref{EQ:EQV4} in the proof of Theorem~\ref{thm:asymptotic_optimal}.  This term is equal to 
\begin{align}
    \frac1T\sum_{t = 0}^{T-1} &\E \left[l(\bX{t}, \bu(t))\right] = \frac1T\sum_{t = 0}^{T-1} \E \left[\rotcost{l}{\mX(t)}{\bu(t)} - \lambda\cdot\mX(t) + \lambda\cdot\Phi(\mX(t), \bu(t)) \right] \nonumber\\
    &= \frac1T\sum_{t = 0}^{T-1} \E \left[ L_\tau(\bX{t}) - L_{\tau-1}(\Phi(\bX{t},\bu(t)) - \lambda\cdot\mX(t) + \lambda\cdot\Phi(\mX(t), \bu(t))  \right]\nonumber\\
    &\le \frac1T\sum_{t = 0}^{T-1} \E \left[ L_\infty(\bX{t}) - L_{\infty}(\Phi(\bX{t},\bu(t))) + 2\epsilon - \lambda\cdot\mX(t) + \lambda\cdot\Phi(\mX(t), \bu(t))  \right],
    \label{EQ:APX_(A)}
\end{align}
where we use the definition of the rotated cost for the first equality, the dynamic principle for the second line and Lemma \ref{lem:C_vs_W} for the last line by choosing $\tau = \tau(\epsilon)$.

\begin{align*}
     \frac1T\sum_{t = 0}^{T-1} [ L_\infty(\bX{t}) - \lambda \cdot\bX{t}] &= \frac1T\sum_{t = 0}^{T-1} [ L_{\infty}(\bX{t+1}) - \lambda \cdot \bX{t+1}]\\
     &\quad+ \frac1T\left(L_{\infty}(\bX{0}) - L_{\infty}(\bX{T}) +\lambda \cdot \bX{0} - \lambda \cdot \bX{T}\right)
\end{align*}
When $T$ goes to infinity, the second line of the above result goes to $0$. This shows that when we take the limit as $T$ goes to infinity, \eqref{EQ:APX_(A)} is equal to
\begin{align*}
    2\epsilon + \lim_{T\to\infty} \frac1T \sum_{t=0}^{T-1} \E \left[\Costmin{\infty}{\mX(t + 1)} - \Costmin{\infty}{\Phi(\mX(t), \bu(t))} - \lambda\cdot\mX(t + 1)  + \lambda \cdot\Phi(\mX(t), \bu(t)) \right],
\end{align*}
which is \eqref{EQ:EQV4}

\subsection{Proof of Theorem~\ref{thm:asymptotic_optimal}: Final computation details}
\label{apx:C5}

In Lemma \ref{lem:C_vs_W}, we show that the function $\Costmin{\infty}{\cdot}$ has the following property in the $l1$ norm: $|\Costmin{\infty}{\bx}-\Costmin{\infty}{\bx'}|\le (\|\lambda\|_{\infty} + k/\rho_k)\|\bx-\bx'\|_1$. Moreover, we show that under appropriate parametrization (see Appendix \ref{APX:LAGBOUND}), the Lagrange multiplier $\|\lambda\|_{\infty}$ is bounded above by $\frac{k}{\syncconst_k}\left( 1 + \frac{\alpha k}{\syncconst_{k} }\right)$. This shows that
\begin{align*}
    |\E\Big[\Costmin{\infty}{\mX(t &+ 1)} - \Costmin{\infty}{\Phi(\mX(t), \bu(t))} - \lambda\cdot\mX(t + 1)  + \lambda \cdot\Phi(\mX(t), \bu(t))\Big]| \\
    &\le (k/\rho_k + 2\|\lambda\|_\infty)\E\left[\|\mX(t + 1) - \Phi(\mX(t), \bu(t)\|_1\right]\\
    &\le \left(\frac{k}{\syncconst_k} + 2\frac{k}{\syncconst_k}\left( 1 + \frac{\alpha k}{\syncconst_{k}}\right)\right)\E\left[\|\mX(t + 1) - \Phi(\mX(t), \bu(t))\|_1\right]\\
    &\le \frac{k}{\syncconst_k} \left( 3 + \frac{2\alpha k}{\syncconst_{k}}\right)\E\left[\|\mX(t + 1) - \Phi(\mX(t), \bu(t))\|_1\right]
\end{align*}
To study the above equation, by the triangular inequality, we have that:
\begin{align*}
    &\E\left[\|\mX(t + 1) - \Phi(\mX(t), \bu(t))\|_1\right]\\
    &\le \E\left[\|\mX(t + 1) - \Phi(\mX(t), \bU(t))\|_1\right] + \E\left[\|\Phi(\mX(t), \bU(t) - \Phi(\mX(t), \bu(t))\|_1\right].
\end{align*}
By Lemma~1 of \cite{GGY23b}, the first term of the above equation is bounded by $\sqrt{|S|}/\sqrt{N}$. Moreover, our rounding procedure implies that $\|\bU(t)-\bu(t))\|_1\le (\alpha N - \lfloor\alpha N\rfloor)/N$. 
As $\Phi$ is the Lipschitz-continuous with a constant\footnote{$C_{\Phi}$ is bounded above by $2$ since the entries of any stochastic matrix are bounded above by $1$ and the transition kernel comprises of the difference of the transition kernel under action $1$ and $0$ in the worst case.
} $2$, this shows that
\begin{align*}
    \E\left[\|\mX(t + 1) - \Phi(\mX(t), \bu(t)\|_1\right]\le \frac{\sqrt{|S|}}{\sqrt{N}} + 2\frac{\alpha N - \lfloor\alpha N\rfloor}{N}.
\end{align*}
Plugging this into \eqref{EQ:EQV4} shows that
\begin{align*}
    (A) \le 2\epsilon + \frac{k}{\syncconst_k} \left( 3 + 2\frac{\alpha k}{\syncconst_{k}}\right) \left(\frac{\sqrt{|S|}}{\sqrt{N}} + 2\frac{\alpha N - \lfloor\alpha N\rfloor}{N} \right).
\end{align*}
By \eqref{EQ:EQV3}, $\Vopt{\mB}{N} - \Vlpt{\mB}{\tau}{N} \le (A) + \frac{\alpha N - \lfloor\alpha N\rfloor}{N}$. Hence: 
\begin{equation*}
   \Vopt{\mB}{N} - \Vlpt{\mB}{\tau}{N} \leq  2\epsilon + \frac{k}{\syncconst_{k}}\left(3 + 2\frac{\alpha k}{\syncconst_{k}} \right)\left(\frac{\sqrt{|\mcl{S}|}}{\sqrt{N}} + \frac{2(\alpha N - \lfloor{\alpha N}\rfloor)}{N}\right) + \frac{(\alpha N - \lfloor{\alpha N}\rfloor)}{N}
\end{equation*}
Theorem~\ref{thm:asymptotic_optimal} holds by grouping the terms that depend on $\frac{(\alpha N - \lfloor{\alpha N}\rfloor)}{N}$ and the terms that depend on $1/\sqrt{N}$.

\section{Proof of Theorem~\ref{thm:expo_bound}}
\label{apx:PFcomp_expo}

\begin{lemma}
    \label{lem:locally_linear}
    Assume~\ref{AS::non-degenerate} and \ref{AS::stable}, then there exists a neighborhood $\mathcal{N}$ of $\bx^*$ and a matrix $A$ such that $\mu_\tau(\bx) = \mu_\tau(\bx^*)+ A(\bx-\bx^*)$ for all $\bx\in\mathcal{N}$.
\end{lemma}
\begin{proof}
    Let us denote by $S^+:=\{i : u^*_i=x^*_i\}$ the set of states for which all the action $1$ is taken for all arms in those state and by $S-+:=\{i : u^*_i=x^*_i\}$ the set of states for which all the action $0$ is taken for all arms in those state.
    Recall that $i^*$ is the unique state such that $0<u^*_i < x^*_i$ (which exists and is unique by Assumption~\ref{AS::non-degenerate}.  We define the function $f:\Delta_\sspace\to\sspace$ by:
    \begin{align*}
        f_i(x) = 
        \left\{\begin{array}{ll}
             x_i & \text{ For all $i\in\sspace^+$}  \\
             \alpha - \sum_{i\in\sspace^+} x_i & \text{ For $i=i^*$}\\
             0 & \text{ For all $i\in\sspace^-$.}
        \end{array} 
        \right.
    \end{align*}
    We claim that there exists a neighborhood $\mathcal{N}$ such that
    \begin{enumerate}
        \item For all $\bx\in\mathcal{N}$, we have $f(\bx)\in\mathcal{U}(\bx)$, which means that $f(\bx)$ is a feasible control for $\bx$. 
        
        \emph{Proof}. We remark that by construction, one has $\sum_i f_i(\bx)=\alpha$. Hence, $f(\bx)$ is a valid control if and only if $0\le\bu\le\bx$. This is clearly true for all $i\ne i^*$. For $i^*$, it is true if $0\le  \alpha - \sum_{i\sspace^+} x_i \le x^*_i$ which holds in a neighborhood of $\bx$ because of non-degeneracy (Assumption~\ref{AS::non-degenerate}) that implies that $0 < u^*_i = \alpha - \sum_{i\sspace^+} x^*_i < x^*_{i^*}$.
        \item There exists a neighborhood $\mathcal{N'}$ of $\bx^*$ such that if $\bx(0)\in\mathcal{N'}$, and we construct the sequence $\bx(t)$ by setting $\bx(t+1) = \Phi(\bx(t), f(\bx(t)))$, then $\bx(t)\in\mathcal{N}$ for all $t$.

        \emph{Proof}. We remark that $\bx(t+1) = \bx^* + (\bx(t)-\bx^*) P^*$, where $P^*$ is the matrix defined in Assumption~\ref{AS::stable}. Indeed, one has:
        \begin{align*}
            \Phi_j(\bx, f(\bx)) &= (\bx P^0)_j + (f(\bx) (P^1 - P^0))_j\\
            &= \sum_{i\in\sspace^+} x_i P^1_{ij} + \sum_{i\in\sspace^-} x_i P^0_{ij} + x_{i^*}P^0_{i^*,j} + (\alpha-\sum_{i\in\sspace^+} x_{i})(P^{1}_{i^*,j}-P^0_{i^*,j})\\
            &= \sum_{i\in\sspace^-\cup\{i^*\}} x_i P^0_{ij} + \sum_{i\in\sspace^+} x_i(P^1_{i,j} - P^{1}_{i^*,j}+P^0_{i^*,j}) + \alpha (P^1_{i^*,j}-P^0_{i^*,j})\\
            &= (\bx P^*)_j + \alpha ( P^1_{i^*,j}-P^0_{i^*,j} ),
        \end{align*}
        where $P^*$ is the matrix defined in Assumption~\ref{AS::stable}.

        Note that by construction, one has $\Phi(\bx^*,f(\bx^*))=\bx^*$. This implies that $\Phi_j(\bx, f(\bx)) - \Phi_j(\bx^*, f(\bx^*))= (\bx-\bx^*)P^*$. The result follows by Assumption~\ref{AS::stable} that imposes that the matrix $P^*$ is stable.
        
        \item Let $Z(\bx)$ be the reward function collected by the control $f(\bx)$. $Z(\bx)$ is a linear function of $\bx$ and $Z(\bx^*)=\tau g^*+\lambda \cdot x^*=\VoptInf{\tau}{\bx^*}$. As a result, $Z(\bx)=\VoptInf{\tau}{\bx}$ for all $\bx\in\mathcal{N'}$. 

        \emph{Proof}. The linearity of $Z$ is a direct consequence of the linearity of $f$ and $\Phi$. The fact that $Z(\bx^*)=\tau g^*+\lambda \cdot x^*$ is because $f(\bx^*)=\bu^*$. It is also the optimal value of $W_\tau(\bx^*)$ by dissipativity. Last, as the function $W_t()$ is the solution of an LP, it is concave. Since $Z$ is linear and $0<x_i<1$ for all state $i$, it follows that the two functions must coincide.
    \end{enumerate}
\end{proof}

\begin{lemma}
    \label{lem:concentration}
    Assume \ref{AS::SA}, and \ref{AS::unique}. Assume that $\alpha N$ is an integer. For any neighborhood $\mathcal{N}$ of $\bx^*$ There exists a $C>0$ such that if $\bX{t}$ is a trajectory of the optimal control problem, then: 
    \begin{align*}
        \lim_{t\to\infty} \Pr[\bX{t}\in\mathcal{N}] \le e^{-C N }.
    \end{align*}
\end{lemma}

\begin{proof}
    We use two main steps:
    \begin{enumerate}
        \item For any initial condition $\bx(0)$, let $\bx_\tau(t+1)=\Phi(\bx_\tau(t), \mu_\tau(\bx_\tau(t)))$, then there exists $\tau$ and $T$ such that for any initial condition $\bx(0)$ and any $t\ge T$, one has $\bx_\tau(t)\in\mathcal{N}$.
        
        \emph{Proof}. The result follows by dissipativity (which holds because of Assumption~\ref{AS::SA}). Indeed, as $\bx_{\tau}(\cdot)$ is an optmal trajectory, it must hold that $\lim_{\tau\to\infty, T\to\infty} \rotcost{l}{\bx_\tau(t)}{\mu_\tau(\bx_\tau(t))} = 0$.
        
        \item The map $\bx\mapsto\mu_\tau(\bx)$ is Lipschitz-continuous in $\bx$. 
        
        \emph{Proof}. By assumption~\ref{AS::unique}, the control $\mu_\tau(\bx)$ is unique. The result then follows because $\mu_\tau(\bx)$ is the solution of a linear program parametrized by $\bx$.
    \end{enumerate}
    The lemma can then be proven by adapting the proof of Theorem~3 of \cite{GGY23b} and in particular their Equation~(17). This is proven by using Lemma~1 \cite{GGY23b} that implies that $\Pr[\|\bX{t+1}-\Phi(\bX{t}, \bU(t))\| \ge \epsilon] \le e^{-C'N}$. 
\end{proof}

We are now ready to prove Theorem~\ref{thm:expo_bound}. We can follow the proof of Theorem~\ref{thm:asymptotic_optimal} up to 
\eqref{EQ:EQV4} that shows that $\Vopt{\mB}{N} - \Vlpt{\mB}{\tau}{N}$ is bounded\footnote{In this expression, the term $(\alpha N - \lfloor{\alpha N}\rfloor)/N$ is equal to $0$ here because we assumed that $\alpha N$ is an integer.} by 
\begin{align*}
      \epsilon + \lim_{\tau\to\infty}  \frac{1}{\tau}\sum_{t = 0}^{\tau - 1} \E \left[\Costmin{\tau}{\mX(t + 1)} - \Costmin{\tau}{\Phi(\mX(t), \bu(t))} - \lambda\mX(t + 1)  + \lambda \Phi(\mX(t), \bu(t)) \right] 
\end{align*}
Let $g(\bx) = L_\tau(\bx) - \lambda \bx$. By Lemma~\ref{lem:locally_linear}, this function is linear on $\mathcal{N}$. Let us denote by $E(t)$ the event
\begin{align*}
    E(t) := \{\bX{t+1}\in \mathcal{N} \land \Phi(\mX(t), \bu(t))\in\mathcal{N}\}.
\end{align*}
Hence, this shows that:    
\begin{align*}
    \E &\left[\Costmin{\tau}{\bX{t + 1}} - \Costmin{\tau}{\Phi(\bX{t}, \bu(t))} - \lambda\bX{t + 1}  + \lambda \Phi(\bX{t}, \bu(t)) \right]\\
    &= \E\left[ g(\bX{t+1}) - g(\Phi(\bX{t}, \bu(t)))\right]\\
    &= \E\left[ (g(\bX{t+1}) - g(\Phi(\bX{t}, \bu(t))))\mathbf{1}_{E(t)} \right] + \E\left[ (g(\bX{t+1}) - g(\Phi(\bX{t}, \bu(t))))\mathbf{1}_{\not E(t)} \right]
\end{align*}
By linearity of the $g$ when $E(t)$ is true and by the fact that $\E[\bU(t)]=\bu(t)$ (Lemma~\ref{apx::lem_round}), the first term is equal to $0$. Moreover, as $E(t)$ is true with probability at least $1-2e^{-CN}$, the second term is bounded by $C' e^{-CN}$ when $t$ is large. 

This concludes the proof of Theorem~\ref{thm:expo_bound}. 

\section{Proof of Lemma \ref{LEM::LLIP}}\label{apx::section_Lip}
\subsection*{Proof Outline for Lemma \ref{LEM::LLIP}}
The first part of this section is dedicated to proving the Lipschitz property for the value function $\Vdetdisc{\bx}{}$ in $\bx$ under the ergodicity Assumption \ref{AS::SA}. For simplicity of exposition we will set $k = 1$, although it is not hard to extrapolate the proof for $k > 1$ and is left as an exercise for the reader.

The key (rather counter-intuitive) idea behind this proof is to rewrite the problem in terms of an $M$ component vector $\bs \in \sspace^{(M)}$. As $M$ tends to infinity we will use the dense nature of rational numbers in the real line and continuity arguments to argue that the proof holds for any $\bx \in \Dels$.

Hence, to start with, we will assume that $\bx \in \Dels^{(M)}$, the set of all points on the simplex that can be represented by an $M$ component vector $\sspace^{(M)}$. In order to complete this proof we will need an intermediate result that verifies the Lipschitz property for all $\bx, \bx' \in \Dels^{M}$ such that $\|\bx  - \bx'\| \leq \frac{2}{M}$ i.e, they differ on at most two components. 

\subsection*{Proof of Lemma \ref{LEM::LLIP}}
Following the proof outline above let $\bx, \bx' \in \Dels^{(M)}$. Hence, there exist unique (up to permutation) $M$ component vectors $\bs(\bx) := \{s_0, s_1 \dots s_{M-1}\}$ and $\bs'(\bx) := \{s'_0, s'_1 \dots s'_{M-1}\}$. 
\begin{lemma}\label{LEM::LLIP2}
    Under assumption \ref{AS::SA}, for any $\beta \in (0, 1]$, let $\bx, \bx' \in \Dels^{(M)}$ with $\|\bx - \bx'\|_1 \leq \frac{2}{M}$, then,
    \begin{equation}
        \Vdetdisc{\bx}{} - \Vdetdisc{\bx'}{} \leq \frac{1}{\syncconst}\|\bx - \bx'\|_1
    \end{equation}
\end{lemma}
\begin{proof}
    Let $\bs(\bx) := \{s_0, s_1 \dots s_{M-1}\}$, WLOG we can assign $\bs'(\bx') := \{s_0', s_1, s_2 \dots s_{M-1}\}$. Note, in this case if $s_0$ is the $1^{\text{st}}$ component ($s_0 = 1$) then one can look at the fraction of arms in state $1$, $\mB(\bs)_{1} := x_{1} $. Now if $s_0' = i \neq 1$, we have $x'_1 = x_1 - \frac{1}{M}$ and $x'_i = x_i + \frac{1}{M}$, allowing us to conclude $\|\bx - \bx'\|_1 = \frac{2}{M}$. Note, for any $\bx$, let $\ba$ be an action vector \textit{i.e.} $\ba = \{a_0, a_1 \dots a_{M-1}\}$. Since, $\Vdetdisc{\bx}{}$ is always well defined for any $\beta < 1$, one can write the $Q$ function for state $\bs$ and action $\ba$ as follows:
    \begin{equation*}
        Q(\bs, \ba) := \frac{1}{M}\sum_{n = 0}^{M - 1}R^{a_n}_{s_n} + \beta \sum_{\bs'' \in \Dels^{(M)}} \Vdetdisc{\mB(\bs'')}{} \Pi_{n = 0}^{M - 1} P^{a_n}_{s_n, s_n''}
    \end{equation*}
    Now, there exists $\ba^{*} := \{a_0, a_1 \dots a_{M - 1}\}$ such that $Q(\bs, \ba^{*}) = \Vdetdisc{\mB(\bs)}{}$. Suppose we pick $\ba' := \{0, a_1, a_2 \dots a_{M - 1}\}$. Note, $\ba'$ always satisfies the constraint on the action space (if needed by pulling one less arm). Hence,
    \begin{align*}
       \Vdetdisc{\bx}{} - \Vdetdisc{\bx'}{} &\leq Q(\bs, \ba^{*}) - Q(\bs', \ba')\\
       =& \frac{R^{a_{0}}_{s_0} - R^{0}_{s_0'}}{M} + \sum_{\bs'' \in \sspace^{(M)}}\Vdetdisc{\mB(\bs'')}{} \Pi_{n = 1}^{M - 1} P^{a_n}_{s_n, s_n''}\left( P^{a_0}_{s_0, s_0''} -  P^{0}_{s_0', s_0''} \right)\\
       =& \frac{R^{a_{0}}_{s_0} - R^{0}_{s_0'}}{M} + \\
       &  \sum_{s_1'' \dots s_{M - 1}''}\Pi_{n = 1}^{M - 1}P^{a_n}_{s_n, s_n''} \Bigg[\sum_{i \in \sspace} \Vdetdisc{\mB(\{i, s_1'' \dots s_{M - 1}''\})}{}(P^{a_n}_{s_n, i} - P^{0}_{s_n', i})\Bigg]
    \end{align*}
    Note, $0 \leq R^{a_{0}}_{s_0} \leq 1$, so the first term is bounded above by $\frac{1}{M}$. We will focus on the second term, to this end let $\syncconst(i) = \min \{P^{a_{0}}_{s_0,i}, P^{0}_{s_0',i}\}$. 
    \begin{align*}
         &\Bigg[\sum_{i \in \sspace} \Vdetdisc{\mB(\{i, s_1'' \dots s_{M - 1}''\})}{}(P^{a_n}_{s_n, i} - P^{0}_{s_n', i})\Bigg]\\
         =&\sum_{i \in \sspace} \Vdetdisc{\mB(\{i, s_1'' \dots s_{M - 1}''\})}{}(P^{a_0}_{s_0, i} - \syncconst(i)) - \sum_{i \in \sspace} \Vdetdisc{\mB(\{i, s_1'' \dots s_{M - 1}''\})}{}(P^{0}_{s_0', i} - \syncconst(i)) \\
         \leq& \max_{i \in \sspace}\Vdetdisc{\mB(\{i, s_1'' \dots s_{M - 1}''\})}{}(1 - \sum_{i \in \sspace}\syncconst(i)) - \min_{i' \in \sspace}\Vdetdisc{\mB(\{i', s_1'' \dots s_{M - 1}''\})}{}(1 - \sum_{i \in \sspace}\syncconst(i))
    \end{align*}
    Note, by assumption (\ref{AS::SA}), $\syncconst := \sum_{i \in \sspace}\syncconst(i) > 0$. Further, let $$\sigma := \max_{\bx, \bx'\in \Dels^{(M)}, \|\bx - \bx'\|_1 \leq 2/M}|\Vdetdisc{\bx}{} - \Vdetdisc{\bx'}{}| $$ 
    Let us denote by $\Tilde{\bx} := \mB(\{i, s_1'' \dots s_{M - 1}''\})$ and $ \Tilde{\bx}' := \mB(\{i', s_1'' \dots s_{M - 1}''\})$ we have: $\|\Tilde{\bx} - \Tilde{\bx}'\|_1\leq 2/M$. This leads us to the following result, for any $\bx, \bx'$ such that $\|\bx - \bx'\|_1 \leq 2/M$ we have:
    \begin{align*}
        |\Vdetdisc{\bx}{} - \Vdetdisc{\bx'}{}| \leq \frac{1}{M} + (1 - \syncconst) \sigma
    \end{align*}
    In particular, choose $\bx, \bx' \in \Dels^{(M)}$ that maximize the value of $\Vdetdisc{\bx}{} - \Vdetdisc{\bx'}{}$, we then have,
    \[
    \sigma \leq \frac{1}{M} + (1 - \syncconst) \sigma
    \]
    Finally, putting the results together we have for any $\bx, \bx' \in \Dels^{(M)}$ with $\|\bx - \bx'\|_1 \leq 2/M$:
    \[
    |\Vdetdisc{\bx}{} - \Vdetdisc{\bx'}{}| \leq \sigma \leq \frac{1}{2 \syncconst} \|\bx - \bx'\|_1
    \]
\end{proof}

We are now ready for the final steps of the proof.

\begin{proof}[Proof of Lemma \ref{LEM::LLIP}]
    The result then follows by noting that for any $\bx, \bx' \in \Dels^{(M)}$ there exists a shortest path from $\bx$ to $\bx'$, say $\{\bx = \bx_{0}, \bx_{1}, \bx_{2} \dots ,\bx_{P} = \bx'\}$ such that between any two sequential components exactly one component is changed at a time $\|\bx_{i} - \bx_{i + 1}\| = \frac{2}{M}$. Clearly, no more than $2 M$ changes can occur along this path, it follows that $P < 2 M$. Summing along this path allows us to see that for all $\bx, \bx' \in \Dels^{(M)}$ we have,
        \[
    |\Vdetdisc{\bx}{} - \Vdetdisc{\bx'}{}| \leq \frac{1}{ \syncconst} \|\bx - \bx'\|_1
    \]
    We can now complete the proof by noting that $\cup_{M > 0}\Dels^{(M)}$ is dense in $\Dels$. Hence, the result must hold for any $\bx, \bx' \in \Dels$. This completes the proof.    
\end{proof}

\section{Bound on Lagrange multiplier $\lambda$} \label{APX:LAGBOUND}

We start by defining the fixed point problem in terms of the state action occupation measure $\by := \{y(s, a)\}_{s, a}$  similar to the choice made in \eqref{EQ::StateAction}, this problem corresponds exactly with problem \eqref{EQ::VOPTDYN-Tinf}. Recall, $y(s, a)$ is the fraction of arms in state $s$ while performing action $a$.The average reward problem is restated as follows :
\begin{align}\label{EQ:PSAC}
    \Vopt{}{\infty} := \max_{\by} \sum_{s, a} y(s, a) r(s, a) \tag{$P_{\text{sac}}$}\\
    \text{s. t.} \nonumber \\
    \sum_{a} y(s, a) = \sum_{s', a} y(s', a) P(s|s', a) \hspace{0.2 in} \text{for all $s$} \nonumber\\
    \sum_{s} y(s, 1) \leq \alpha \nonumber \\
    \sum_{s, a} y(s, a) = 1 \nonumber
\end{align}

We note here that this problem can equivalently be seen as an average reward problem for a single arm markovian bandit with a relaxed action constraint, $\alpha$ which can be treated as an \emph{average time constraint} rather than a \emph{per time} constraint problem. Concretely, we are enforcing the following constraint
\[
\lim_{T \uparrow \infty}\frac{1}{T} \sum_{t = 0}^{T- 1}\sum_{s} y(s, 1, t) \leq \alpha \hspace{0.1 in} \text{equivalently} \hspace{0.1 in} 
\lim_{T \uparrow \infty}\sum_{t = 0}^{T - 1} \|u(t)\|_1 \leq \alpha \hspace{0.2 in}  \hspace{0.2 in} 
\]
where $y(s,1,t)$ will then correspond to the probability that the arm is in state $s$ and play action $a$ at time $t$ under a fixed policy.

Let $\laghomosub \geq 0$ be a multiplier corresponding to the  $(\sum_{s} y(s, 1) -\alpha ) \leq 0$ constraint. One can then formulate a new problem with a cost for playing $1$ equal to $\laghomosub$. Hence, we have:

\begin{align}\label{EQ:RELPROB}
    \max_{\by} \sum_{s, a} y(s, a) r(s, a) + \laghomosub(\alpha - \sum_{s} y(s, 1))\tag{$P_{\text{sa}}(\laghomosub)$}\\
    \text{s. t.} \nonumber \\
    \sum_{a} y(s, a) = \sum_{s', a} y(s', a) P(s|s', a) \hspace{0.2 in} \text{for all $s$} \nonumber\\
    \sum_{s, a} y(s, a) = 1 \nonumber
\end{align}

This problem can be reformulated in terms of a new reward function $r_{\laghomosub}(s, a)$ defined by $r_{\laghomosub}(s, 0) = r(s, 0)$ and $r_{\laghomosub}(s, 1) = r(s, 1) - \laghomosub$ to yield:
\begin{align}
    \max_{\by} \sum_{s, a} y(s, a) r_{\laghomosub}(s, a) + \laghomosub\alpha \\
    \text{s. t.} \nonumber \\
    \sum_{a} y(s, a) = \sum_{s', a} y(s', a) P(s|s', a) \hspace{0.2 in} \text{for all $s$} \nonumber\\
    \sum_{s, a} y(s, a) = 1 \nonumber
\end{align}

Consider the following dual of \eqref{EQ:RELPROB}, the dual variables are given by a \emph{bias vector} $[\biasvec{s}]$ and a \emph{gain value} $\gainrel$ such that:
\begin{align}\label{EQ:DUALRELPROB}
    \min_{[\biasvec{s}], \gainrel} \gainrel& + \laghomosub \alpha \tag{$DP(\laghomosub)$}\\
    &\text{s. t.} \nonumber \\
    \gainrel + \biasvec{s}& - \sum_{s'} P(s'|s, a) \biasvec{s'} \geq r_{\laghomosub}(s, a) \hspace{0.1 in} \text{for all $s, a$}\nonumber
\end{align}

\begin{remark}
    Note, we refrain from using $h^{\star}(s)$ to denote the bias for this problem and opt instead to use a separate notation, $\biasvec{s}$ even though the two fixed point equations bear a resemblance to each other. This is because $h^{\star}(s)$ is the bias for the mean field limit problem whose domain lies in the space of distributions over arm states i.e, $h^{\star}(s)$ is a function whereas $\biasvec{s}$ is simply a vector. Indeed, if we knew the exact function, $h^{\star}(s)$ apriori then there would be no need to solve a $T$ horizon problem and we could instead solve a single step dynamic program and set the terminal cost to $h^{\star}(\cdot)$ to find the optimal policy at each step. 
\end{remark}

The following is a rather straight-forward result that equates the problem \eqref{EQ:RELPROB} to \eqref{EQ:DUALRELPROB} and \eqref{EQ:PSAC}.

\begin{lemma}\label{LEM:DUAL}
    Given a fixed multiplier $\laghomosub$, the programs \eqref{EQ:DUALRELPROB} and \eqref{EQ:RELPROB} are equivalent. Further, there exists vector $\biasvec{\cdot}$ that satisfies the following fixed point equations:
    \[
    \gainrel + \biasvec{s} = \max_{a} r_{\laghomosub}(s, a) + \sum_{s'} P(s'|s, a) \biasvec{s'}
    \]
    with
    \[
    \gainrel := \max_{\by} \sum_{s, a} y(s, a) r_{\laghomosub}(s, a) 
    \]
    . Finally, there exists $\laghomosub$ so that the programs \eqref{EQ:DUALRELPROB} and \eqref{EQ:PSAC} are equivalent.  
\end{lemma}
\begin{proof}
    By Definition 11.4,  \cite{altman1999constrained} our MDP is finite and hence, contracting with a finite absorbing set. Setting the absorbing set to the whole state space, we find that any constant function greater than $1$ is a uniform Lyapunov function for our MDP. Now by Theorem 12.4, \cite{altman1999constrained}, \eqref{EQ:DUALRELPROB} and \eqref{EQ:RELPROB} are equivalent. 
    The last result holds as a direct consequence of strong duality of the linear program \eqref{EQ:PSAC}. 
\end{proof}
     An immediate consequence of the result  above is a bound on the Lagrange multiplier $[\biasvec{s}]$. The span norm of a vector $\mathbf{v}$, denoted by $\|\mathbf{v}\|_{sp} := \max_{s_1}\mathbf{v}(s_1) - \min_{s_2}\mathbf{v}(s_2)$. The span norm is a semi norm and is often used to show convergence in value iteration, see for example \cite{puterman2014markov}. 
\begin{lemma}\label{LEM::BIASBOUND}
    The Lagrange multiplier $[\biasvec{s}]$ is bounded in the span norm as follows:
    \[
    \|[\biasvec{\cdot}]\|_{sp} \leq \frac{k \|\|r\|_{sp} + \laghomosub \alpha \|_{sp}}{\syncconst_{k}}
    \]
\end{lemma}
\begin{proof}
For the purposes of this proof, we are going to assume that the ergodicity coefficient is non-zero at $1$ i.e,  $\syncconst{k} > 0$ occurs when $k = 1$, denoted by $\syncconst_{1}$. However, this proof can easily be extended for any fixed $k$. Given any two states $s_1$ and $s_2$ we have,
   \begin{align*}
       \biasvec{s_1} - \biasvec{s_2} =& \max_{a_1} r_{\laghomosub}(s_1, a_1) + \sum_{s'} P(s'|s_1, a_1) \biasvec{s'}\\
       &\hspace{0.2 in}- \max_{a_2} \left[r_{\laghomosub}(s_2, a_2) + \sum_{s'} P(s'|s_2, a_2) \biasvec{s'}\right]\\
       \leq& \max_{a_1} \left[r_{\laghomosub}(s_1, a_1) + \sum_{s'} P(s'|s_1, a_1) \biasvec{s'}\right]\\
       &\hspace{0.2 in}- \left[r_{\laghomosub}(s_2, 0) + \sum_{s'} P(s'|s_2, 0)\biasvec{s'}\right]\\
       \leq& \|r_{\laghomosub}\|_{sp} + \max_{a_1}\sum_{s'}\biasvec{s'}\left(P(s'|s_1, a_1) - P(s'|s_2, 0)\right)\\
       =& \|r_{\laghomosub}\|_{sp} + \max_{a_1}\sum_{s'}\biasvec{s'}\bigg[(P(s'|s_1, a_1) - \rho_{s'}) - (P(s'|s_2, 0) - \rho_{s'})\bigg]\\
       \leq& \|r_{\laghomosub}\|_{sp} + \max_{s}\biasvec{s} \sum_{s'} (P(s'|s_1, a_1) - \rho(s'))\\
       &\hspace{0.4 in}- \min_{s}\biasvec{s} \sum_{s'}(P(s'|s_2, 0) - \rho(s'))\\
       \leq& \|r_{\laghomosub}\|_{sp} + \|[\biasvec{\cdot}]\|_{sp}(1 - \syncconst_{1})
   \end{align*}
   Here, $\rho(s') := \min\{P(s'|s_1, a_1), P(s'|s_2, 0)\}$ and the last step follows from the definition of the ergodic coefficient. Note that $\|r_{\laghomosub}\|_{sp} \leq ( \|r\|_{sp} + |\laghomosub| \alpha)$. Finally, noting that the Left hand side holds for all pairs $s_1, s_2$ completes the result. 
\end{proof}

It should be noted that the vector of multipliers, $\biasvec{\cdot}$ can be shifted by an arbitrary constant and still provide identical results. Therefore, our choice of using the span norm to bound the multipliers is appropriate since the span norm is invariant when the vector is shifted by a constant value unlike other norms.
Finally, we will need a bound on $\laghomosub$ in order to complete our bound.

\begin{lemma}\label{LEM:LAGBOUND}
    The absolute value of the multiplier for the action constraint, $\laghomosub$ can be bounded by $\frac{k\|r\|_{sp}}{\syncconst_{k}}$
\end{lemma}
\begin{proof}
    Once again our proof will use $k = 1$ but this proof can easily be extended for any $k$. Note that, for some $\nu > \lagmax$, we will find that the optimal action to play for all states is $0$. Hence, at $\nu = \lagmax$ we can choose one distinguished state $s_0$ where the actions $0$ and $1$ give the same value under the optimal policy and play $0$ for all other states. Since, $\laghomosub \leq \lagmax$, our bound proceeds by computing a bound on $\lagmax$. We will denote by $\biasgen{\cdot}{\lambda}$ the multipliers for the corresponding dual problem \eqref{EQ:DUALRELPROB} under reward $r_{\lambda}$. Now we note, by the statement above we have the following,
    \begin{align*}
        \gaingen{\lagmax} + \biasgen{s_0}{\lagmax} =& r(s_0, 0) + \sum_{s'}\biasgen{s'}{\lagmax}P(s'|s_0, 0)\\
        =& r(s_0, 1) - \lagmax + \sum_{s'}\biasgen{s'}{\lagmax}P(s'|s_0, 1)
    \end{align*}
    This yields the following inequality on $\lagmax$,
    \begin{align*}
        \lagmax =& r(s_0, 1) - r(s_0, 0) +  \sum_{s'}\biasgen{s'}{\lagmax}\bigg(P(s'|s_0, 1) - P(s'|s_0, 0) \bigg)\\
        \leq&  \|r\|_{sp} + (1 - \syncconst_1)\|\biasgen{\cdot}{\lagmax}\|_{sp} 
    \end{align*}  
    Finally, let us bound $\|\biasgen{\cdot}{\lagmax}\|_{sp}$. We now have,
    \[
    \gaingen{\lagmax} + \biasgen{s_1}{\lagmax} = r(s_1, 0) + \sum_{s'}\biasgen{s'}{\lagmax}P(s'|s_1, 0)
    \]
    and
    \[
    \gaingen{\lagmax} + \biasgen{s_2}{\lagmax} = r(s_2, 0) + \sum_{s'}\biasgen{s'}{\lagmax}P(s'|s_2, 0)
    \] 
    Hence,
    \begin{align*}
     \biasgen{s_1}{\lagmax} - \biasgen{s_2}{\lagmax} =& r(s_1, 0) - r(s_2, 0) + \sum_{s'}\biasgen{s'}{\lagmax}\bigg(P(s'|s_1, 0) - P(s'|s_2, 0)\bigg) \\
     \leq& \|r\|_{sp} + (1 - \syncconst_1)\|\biasgen{\cdot}{\lagmax}\|_{sp}
    \end{align*}
    In particular,
    \[
    \|\biasgen{\cdot}{\lagmax}\|_{sp} \leq \|r\|_{sp} + (1 - \syncconst_1)\|\biasgen{\cdot}{\lagmax}\|_{sp}
    \]
    Hence, $\|\biasgen{\cdot}{\lagmax}\|_{sp} \leq \frac{\|r\|_{sp}}{\syncconst_1}$. Substituting this value to our bound on $\lagmax$ yields the result of the lemma. 
\end{proof}
The same process can be repeated for $\nu_{\text{min}}$ to get similar lower bounds on $\laghomosub$.

Combining the results of lemmas \ref{LEM::BIASBOUND} and \ref{LEM:LAGBOUND}, noting that $ \|r\|_{sp} = 1$, we find the following bound on $[\biasvec{\cdot}]$.

\begin{corollary}
    The lagrange multiplier is bounded in the span norm as follows,
    \begin{equation}
        \|[\biasvec{\cdot}]\|_{sp} \leq \frac{k }{\syncconst_{k}}\left(1 + \frac{k \alpha}{\syncconst_{k}} \right)
    \end{equation}
\end{corollary}
 Finally, since the Lagrange multiplier is invariant under constant shifts, we will specify $\lambda$ as used in Theorem \ref{thm:asymptotic_optimal} as the unique vector such that all the components are greater than or equal to $0$, with the lowest value being equal to $0$ i.e, $\lambda(s) = \biasvec{s} - \biasvec{s_{\text{min}}} \geq 0$, where $s_{\text{min}} := \arg\min \biasvec{s}$, the vector $\lambda (s)$ so constructed allows us to conclude that 
 \[
 \|\lambda\|_{\infty} = \|[\biasvec{s}]\|_{sp} \leq \frac{k }{\syncconst_{k}}\left(1 + \frac{k \alpha}{\syncconst_{k}} \right)
 \]
as desired. 

\section{Details about the experiments}
\label{apx:parameters}

\subsection{Generalities}

All our simulations are implemented in \texttt{Python}, by using \texttt{numpy} for the random generators and array manipulation and \texttt{pulp} to solve the linear programs. All simulations are run on a personal laptop (macbook pro from 2018). To ensure reproducibility, we will make the code and the Python notebook publicly available (this not done now for double-blind reasons). 

\paragraph{Value of $\tau$} -- Except specified otherwise, we use the value $\tau=10$ in all of our example except for the example \cite{chen:tel-04068056} where we use $\tau=50$.

\paragraph{Number of simulations and confidence intervals} --  To obtain a estimate of the steady-state average performance, we simulate the system up to time $T=1000$ and estimate the average performance by computing the average over values from $t=200$ to $T=1000$.  All reported confidence intervals correspond to 95\% interval on the mean computed by using $\hat{\mu}\pm 2\frac{\hat{\sigma}}{\sqrt{K-1}}$, where $\hat{\mu}$ and $\hat{\sigma}$ are the empirical mean and standard deviation and $k$ is the number of samples (the number of independent trajectories and/or number of randomly generated examples).

\paragraph{A note on choice of algorithms for comparison} As specified in the main body of the paper, we choose to restrict our comarison to two algorithms: LP-update and FTVA.  We note here that \citet{HXCW24} is theoretically an exponentially optimal solution to the RMAB problem and its absence in our comparison might seem conspicuous to the reader. \citet{HXCW24} is a \textit{two set policy} which requires the \emph{identification of a largest set} that can be aligned with the fixed point and actively steering the remaining arms towards this set. Theoretically, we know that for a suitable definition such a set always exists for any finite joint state, however, from a practical perspective the problem of choosing a ``largest'' set that aligns with the fixed point seems unclear to us. Unfortunately, the authors do not provide an implementation for their algorithm and hence, we refrain from comparing our work in order to avoid misleading our readers about the efficacy of their algorithm in comparison to our own as the performance of such an algorithm might heavily depend on the choice of hyperparameter used to implement the alignment procedure. On the contrary, the LP-update does not have any hyperparameters (except for $\tau$ that we set to $\tau=10$ in all our simulations), and so have FTVA and LP-update.

\subsection{Example \cite{HXCW23}}

This example has $8$ dimensions and its parameters are:
\begin{align*}
    P^{ 0 }=\left(
    \begin{array}{cccccccc}
       1 & & & & & & & \\
       1 & & & & & & & \\
        &0.48 &0.52 & & & & & \\
        & &0.47 &0.53 & & & & \\
        & & & &0.9 &0.1 & & \\
        & & & & &0.9 &0.1 & \\
        & & & & & &0.9 &0.1 \\
       0.1 & & & & & & &0.9 \\
    \end{array}\right)
\end{align*}

\begin{align*}
    P^{ 1 }=\left(
    \begin{array}{cccccccc}
       0.9 &0.1 & & & & & & \\
        &0.9 &0.1 & & & & & \\
        & &0.9 &0.1 & & & & \\
        & & &0.9 &0.1 & & & \\
        & & &0.46 &0.54 & & & \\
        & & & &0.45 &0.55 & & \\
        & & & & &0.44 &0.56 & \\
        & & & & & &0.43 &0.57 \\
    \end{array}\right)
\end{align*}

$R^{ 1 }_i = 0 $ for all state and $R^{ 0 } =  (0, 0,0 ,0 ,0 ,0 ,0 ,0.1)$.

For $\alpha=0.5$ (which is the parameter used in \cite{HXCW23}), the value of the relaxed LP is $0.0125$ and the LP index are $[0.025 , 0.025 , 0.025 , 0.025 , 0 , -0.113 , -0.110 , -0.108]$. Note that these numbers differ from the ``Lagrangian optimal indices'' given in Appendix~G.2 of \cite{HXCW23}. This is acknowledge in \cite{HXCW23} when the authors write ``In [their] setting, because the optimal solution $y$ remains optimal even without the budget constraint, we can simply remove the budget constraint to get the Lagrangian relaxation [...] A nuance is that the optimal Lagrange multiplier for the budget constraint is not unique in this setting, so there can be different Lagrange relaxations''. As a result, the priorities given by their indices and the true LP-index are different. In our simulation, we use the ``true'' LP-index and not their values. We also tested their values and obtained results that are qualitatively equivalent for the LP-priority (\emph{i.e.}, for their order or ours, the LP-priority essentially gives no reward).

\subsection{Example \cite{chen:tel-04068056}}

This example has $|\sspace|=3$ dimensions and its parameters are:

\begin{align*}
    P^{ 0 }=\left(
    \begin{array}{ccc}
       0.022 &0.102 &0.875 \\
       0.034 &0.172 &0.794 \\
       0.523 &0.455 &0.022 \\
    \end{array}\right)\qquad 
    P^{ 1 }=\left(
    \begin{array}{ccc}
       0.149 &0.304 &0.547 \\
       0.568 &0.411 &0.020 \\
       0.253 &0.273 &0.474 \\
    \end{array}\right)
\end{align*}

$R^{ 1 } = ( 0.374, 0.117, 0.079) $ and $R^0_i=0$ for all $i\in\sspace$. 

For $\alpha=0.4$ (which was the parameter used in \cite{chen:tel-04068056}), the value of the relaxed LP is $0.1238$ and its LP-index are $(0.199 , -0.000 , -0.133)$.

\subsection{Random example}

To generate random example, we use functions from the library \texttt{numpy} of \texttt{Python}: 
\begin{itemize}
    \item To generate the transition matrices, We generate an array of size $S\times2\times S$ by using the function \texttt{np.random.exponential(size=(S, 2, S))} from the library numpy of python and we then normalize each line so that the sum to $1$. 
    \item We generate reward by using \texttt{np.random.exponential(size=(S, 2))}
\end{itemize}
The example used in Figure~\ref{fig:perf_function_of_N}(a) corresponds to setting the seed of the random generator to $3$ by using \texttt{np.random.seed(3)} before calling the random functions. Its parameters (rounded to $3$ digits) are:

\begin{align*}
    P^{ 0 }=\left(
    \begin{array}{cccccccc}
       0.101 &0.155 &0.043 &0.090 &0.281 &0.285 &0.017 &0.029 \\
       0.006 &0.207 &0.076 &0.136 &0.085 &0.299 &0.147 &0.043 \\
       0.317 &0.254 &0.065 &0.013 &0.144 &0.111 &0.061 &0.035 \\
       0.098 &0.183 &0.069 &0.068 &0.218 &0.028 &0.200 &0.136 \\
       0.053 &0.080 &0.009 &0.038 &0.483 &0.036 &0.159 &0.143 \\
       0.018 &0.105 &0.027 &0.397 &0.150 &0.102 &0.161 &0.040 \\
       0.110 &0.050 &0.088 &0.024 &0.023 &0.142 &0.169 &0.393 \\
       0.055 &0.043 &0.017 &0.494 &0.227 &0.034 &0.119 &0.011 \\
    \end{array}\right)
\end{align*}

$R^{ 0 } = (0.073, 0.087, 0.778, 0.186, 1.178, 0.417, 1.996, 1.351)$

\begin{align*}
    P^{ 1 }=\left(
    \begin{array}{cccccccc}
       0.011 &0.124 &0.006 &0.131 &0.224 &0.070 &0.241 &0.191 \\
       0.071 &0.138 &0.033 &0.023 &0.045 &0.250 &0.339 &0.101 \\
       0.093 &0.113 &0.056 &0.061 &0.109 &0.351 &0.157 &0.059 \\
       0.158 &0.176 &0.151 &0.150 &0.060 &0.142 &0.053 &0.109 \\
       0.370 &0.185 &0.261 &0.020 &0.022 &0.064 &0.047 &0.030 \\
       0.199 &0.139 &0.099 &0.050 &0.141 &0.104 &0.082 &0.187 \\
       0.214 &0.088 &0.011 &0.075 &0.295 &0.174 &0.075 &0.068 \\
       0.028 &0.157 &0.126 &0.078 &0.039 &0.127 &0.376 &0.069 \\
    \end{array}\right)
\end{align*}

$R^{ 1 } = (0.059, 3.212, 1.817, 0.302, 2.259, 0.067, 0.344, 0.172)$.

For $\alpha=0.5$, the value of the relaxed problem is $1.3885$ and the LP-index are $0.377 , 3.273 , 0.846 , -0.116 , 0.802 ,  , -1.230 , -0.562$.

For the other randomly generated examples, we compute the average performance over 50 examples by varying the seed between $0$ and $49$ for most of the figures except for Figures~\ref{fig:gain_function_parameters}(b) where we only use $20$ examples (and vary the seed between $0$ and $19$) to improve computation time.

\end{document}